\documentclass{article}
\usepackage{latexsym,amsfonts,amsthm,amsmath,mathrsfs}
\usepackage{color}
 \catcode`@=11
\numberwithin{equation}{section}
\newtheorem{Theorem}{Theorem}[section]
\newtheorem{Proposition}{Proposition}[section]
\newtheorem{Lemma}{Lemma}[section]
\newtheorem{Corollary}{Corollary}[section]
\theoremstyle{definition}
\newtheorem{Assumptions}{Hypotheses}[section]
\newtheorem{Assumption}{Hypothesis}[section]

\theoremstyle{definition}
\newtheorem{Definition}{Definition}[section]
\newtheorem{rem}{Remark}[section]

\def\R{\mathbb{R}}

\def\div{\text{div}}

\def\deg{\mathop{\rm deg}\nolimits}

\def\to{\rightarrow}

\newcommand{\uk}{\underline{k}}
\newcommand{\ok}{\overline{k}}

\newcommand{\eu}{u_{\varepsilon}}
\newcommand{\ev}{v_{\varepsilon}}

\newcommand{\Int}{\iint_{Q_T}}
\newcommand{\diver}{\operatorname{div}}

\title{Non-trivial, non-negative periodic solutions of a system of singular-degenerate parabolic equations with nonlocal terms}
\author{{\sc G. Fragnelli}\\
Dipartimento di Matematica\\ Universit\`{a} di Bari "Aldo Moro"\\
Via
E. Orabona 4\\ 70125 Bari - Italy\\ email: genni.fragnelli@uniba.it\\
{\sc D. Mugnai }\\
Dipartimento di Matematica e Informatica\\Universit\`a di
Perugia\\Via Vanvitelli 1, 06123
Perugia - Italy\\ email:  mugnai@dmi.unipg.it\\
{\sc P. Nistri and D. Papini}\\
Dipartimento di Ingegneria Informatica e Scienze Matematiche\\ Universit\`{a} di Siena\\
Via Roma 56\\ 53100 Siena - Italy\\ email:
 pnistri@dii.unisi.it, papini@dii.unisi.it}

\date{}
\begin{document}

\maketitle

\begin{abstract} We study the existence of non-trivial, non-negative
periodic solutions for systems of singular-degenerate parabolic
equations with nonlocal terms and satisfying Dirichlet boundary
conditions. The method employed in this paper is based on the
Leray-Schauder topological degree theory. However, verifying the
conditions under which such a theory applies is more involved due to
the presence of the singularity. The system can be regarded as a
possible model of the interactions of two biological species sharing
the same isolated territory, and our results give conditions that
ensure the coexistence of the two species.
\end{abstract}


\noindent 2000AMS Subject Classification: 35K65, 35B10, 47H11
\section{Introduction}
In this paper we consider a periodic system of singular-degenerate
parabolic equations with delayed nonlocal terms and Dirichlet
boundary conditions of the form
\begin{equation}\label{1}
\begin{cases}
\begin{aligned}
u_t-\div(|\nabla u^m|^{p-2}\nabla
u^m)\!\!=\!\left(a(x,t)-\!\int_{\Omega}\!K_1(\xi,t)u^2(\xi, t-\tau_1)d\xi\!+\!\int_{\Omega}\!K_2(\xi,t)v^2(\xi, t-\tau_2)d\xi\right)u^{p-1}, \text{in } Q_T\\
v_t-\div(|\nabla v^n|^{q-2}\nabla
v^n)\!\!=\!\left(b(x,t)+\!\int_{\Omega}\!K_3(\xi,t)u^2(\xi, t-\tau_3)d\xi-\!\int_{\Omega}\!K_4(\xi,t)v^2(\xi, t-\tau_4)d\xi\right)v^{q-1},   \text{in } Q_T\\
\end{aligned}\\
u(x, t) = v(x,t)=0,\qquad\qquad \qquad \qquad\qquad \qquad\qquad  \;\; \text{for }(x,t) \in  \partial \Omega \times (0,T), \\
u(\cdot,0)= u(\cdot,T)\text{ and }v(\cdot,0)= v(\cdot,T), \\
\end{cases}
\end{equation}
and we look for continuous weak solutions. Here $\Omega$ is an open
bounded domain of $\R^N$ with smooth boundary $\partial \Omega$,
satisfying the property of positive geometric density, see
\cite{lsu}, $Q_T:= \Omega \times (0,T), T>0$,
$\tau_i\in(0,+\infty)$, the functions $K_i,a,$ and $b$ belong to $
L^\infty(Q_T)$. The exponents $p$ and $ q $ belong to the interval
$(1,2)$, $m> p$, $n > q$ and $s^m=|s|^{m-1}s$. Setting $Au:=
\div(|\nabla u^m|^{p-2}\nabla u^m)$ and $l:= (m-1)(p-1)$, the
operator $Au$ becomes $m^{p-1}\div(|u|^l|\nabla u|^{p-2}\nabla u)$,
which is the operator considered by Ivanov in \cite{i1}, \cite{i2} and \cite{i3}.
According to the classification proposed in these
papers, we say that the first equation in (\ref{1}) is of
\begin{enumerate}
\item \textit{slow diffusion} type if $m > \frac{1}{p-1}$,
\item \textit{normal diffusion} type if $m = \frac{1}{p-1}$,
\item \textit{fast diffusion} type if $m < \frac{1}{p-1}$.
\end{enumerate}
Of course, analogous definition in terms of $n$ and $q$ can be given
for the second equation in \eqref{1}.

The aim of this paper is to extend the results of \cite{fnp} and
\cite{fnp1}, concerning the existence of non-negative, non-trivial
periodic solutions, to a system of singular-degenerate parabolic
equations.
To the best of our knowledge, this is the first result
for the case when $ 1 < p, q < 2 $, $m>p$ and $ n>q$, also in the case of
a single equation.
We recall that the cases $ p, q > 2$,  $ m, n > 1 $ and $ p = q = 2$, $ m, n > 1 $
were treated respectively in \cite{fnp} and \cite{fnp1},
see also \cite{F} and
\cite{FA} for a system of anisotropic $ ( p(x), q(x) ) $-Laplacian parabolic equations,
with $ p(x), q(x) > 2 $ in $ \overline{\Omega} $, and $ m = n = 1 $.
In the very recent paper \cite{WYK}, the authors replace the nonlocal terms of \eqref{1}
by $ \int_{\Omega} K_{i}( \xi, t ) u( \xi, t ) d\xi $, for $ i = 1, 3 $, and
$ \int_{\Omega} K_{i}( \xi, t ) v( \xi, t ) d\xi $, for $ i = 2, 4 $.
By means of local conditions, different from those proposed in \cite{fnp, fnp1},
the authors obtain the coexistence of the two species via a similar topological approach
when $ p, q \ge 2 $, $ m, n \ge 1 $ and, thus, only the slow and normal diffusion occurs, i.e.
$ m ( p - 1 ) \ge 1, n ( q - 1 ) \ge 1 $.
More precisely, models for the interaction between two biological species sharing the same isolated territory, with the interactions
represented by means of the kernels $ K_{i}, i = 1, 2, 3, 4 $, were considered in related systems of doubly
degenerate parabolic equations in \cite{fnp}, \cite{WYK} and in systems of porous medium equations in \cite{fnp1}.
On the other hand, some previous biological models found in the
literature, see e.g. \cite{amn}, \cite{ar}, \cite{os}, \cite{oss}
involve the $p$-Laplacian with $p>1$ (and $m=1$).
Furthermore, we observe that the equations of the system we consider treat all the
possible types of diffusion: slow, normal and fast, while in
\cite{fnp}, \cite{WYK} and \cite{fnp1} only the slow and normal diffusions were
presented.
In fact, as it can be easily checked, if $ p \in \left( 1, \frac{ 1 + \sqrt{5} }{ 2 } \right) $ one can have all the three types of diffusion,
while if $ p \in \left[ \frac{ 1 + \sqrt{5} }{ 2 }, 2 \right) $ only the slow diffusion is possible under the condition $ m > p $.

In the case of the fast diffusion and superlinear growth in $ u, v $ of the right hand sides,
the solutions may blow up or vanish in some finite time depending on the initial conditions
as illustrated in \cite{bb}, \cite{LeRoux}, \cite{LeRouxMainge}, \cite{md} and the
references therein for the simple equation obtained from the first of system \eqref{1} by letting $ p = 2 $,
$ a > 0 $ constant and all the kernels $ K_{i} \equiv 0 $
(observe that in the case when $ p = q = 2 $ no restriction on $ m, n $ are required, see Remark~\ref{rem:mn}).
If $ \Omega = \R^{N} $ for such equation we have that the solution blows up for any initial condition in the case when
the superlinear growth in $ u $ is less than a certain critical exponent, see \cite{md}, and the same occurs for doubly degenerate parabolic equations, see \cite{MuZhengLiu}.
If the growth is linear or sublinear we do not have blow up of any solution, see \cite{LeRouxMainge},
hence solutions exist for all $ t \ge 0 $ and in the linear case, depending on the initial condition, they may vanish in finite time
or become unbounded as $ t \to +\infty $ and, thus, the considered initial conditions cannot give rise to a periodic solution.

The choice of the sublinear exponents $ p - 1 $ and $ q - 1 $, respectively, for $ u $ and $ v $ in \eqref{1} is mainly technical since it depends
on the topological method employed in the paper, which is based on a priori bounds of the solutions.
Indeed, this choice enable us to establish the required a priori bounds on the solutions of the approximating problem \eqref{2} in a uniform way with respect to
the perturbation parameter $ \epsilon > 0 $.
We remark that, if the diffusion is slow and $ \Omega \subset \R^{N}$ is a bounded and open domain,
then we can allow a superlinear growth in $ u, v $ in order to have both global existence solutions and periodic solutions,
together with their $ L^{\infty} $-estimates, see \cite{AnDi}, \cite{c}, \cite{cr}, \cite{m}, \cite {n},
\cite{ohara}, \cite{wgs}, \cite{WYK}, \cite{wyw1}, \cite{wyw2}, \cite{wyw} and the references therein.
Due to the singularity of the $ p, q $-Laplacian, the way of proving the a priori bounds deeply differs from that employed in \cite{fnp} and \cite{fnp1}.
Moreover, in order to pass from the $ L^{2} $- to the $ L^{\infty} $-estimates, in Lemma~\ref{limitatezza} we have readapted Moser's technique to the case
when $ 1 < p, q < 2 $.
Moreover we have to impose the technical restriction $ m > p $ and $ n > q $ in order to get the gradient estimates in Lemma~\ref{stima-gradiente}.

Due to their importance in different physical and other natural
sciences such as non-Newtonian fluid mechanics, flow in porous
medium, nonlinear elasticity, glaciology, population biology etc..,
degenerate and singular parabolic equations have been the subject of
extensive research in the last $25$ years, with particular emphasis
on the study of regularity for non-negative weak solutions. We
mention here, among many others, the papers \cite{i1},
\cite{i2}, \cite{i3}, \cite{pv} and the monographs \cite{db}, \cite{v}. In
particular,  we refer to the very recent monograph \cite{dbgv} for a
comprehensive treatment of the Harnack inequality for non-negative
solutions to $p$-Laplacian and porous medium equations. Moreover,
\cite{dbgv} provides an historical presentation of the achievements
in this research field and many references to the applications
concerning the topics mentioned above.

The regularity results for the singular $p$-Laplacian are crucial
for the application of the topological degree approach used in this
paper. Similar topological methods are also employed to a great
extent for the existence of non-negative periodic solutions of
degenerate and doubly degenerate parabolic equations, see
\cite{afnp}, \cite{badii}, \cite{fm}, \cite{hpt}, \cite{hwk},
\cite{khs}, \cite{lsw}, \cite{Liu}, \cite{m}, \cite{n}, \cite{wywen}, \cite{wg}, \cite{wy},
\cite{wy1}, \cite{WYK}, \cite{wyw1}, \cite{wyw2}, \cite{wyw},
\cite{wwy}, \cite{zhwk}, \cite{zkwy}. Nonlocal models to study
aggregation in biological systems with degenerate diffusion are
proposed in several papers, see the recent \cite{bs},
\cite{li-zhang} and the references therein.

Moreover, we recall that the interest in studying the existence of
periodic solutions for degenerate and non-degenerate parabolic
equations modelling biological and physical phenomena, relies in the
consideration that the periodic behavior of certain biological and
physical non-negative quantities is the most natural and  desirable
one, see e.g. \cite{afnp}, \cite{AlMoVi}, \cite{an}, \cite{AlPa},
\cite{F}, \cite{FA}, \cite{fnp}, \cite{fnp1}, \cite{Hess},
\cite{hwk}, \cite{Jin}, \cite{lsw}, \cite{o}, \cite{wg}, \cite{wgs}, \cite{WYK},
\cite{zhwk}, \cite{zkwy}. We also recall the related problems faced
in \cite{fm1} and \cite{fm2} also for higher order operators, and in \cite{DuHsu} for $p=2$ and $N=1$.

The paper is organized as follows.
The goal of Section~\ref{regularizedproblem} is the proof of a coexistence result based on the explicit knowledge of suitable a priori bounds on the $L^2$-norms of the solutions.
The search for such bounds is carried on in Section~\ref{sec:aprioribounds}.
The reason to split the argument in this way lies in the fact that our main coexistence conditions, namely Hypotheses~\ref{ipotesi1}, are applicable regardless of any other assumption on the terms of the equations.
On the other hand, a priori bounds for the periodic solutions are more easily obtained when we focus on specific situations like the competitive (i.e. $ K_{2}, K_{3} \le 0 $) and cooperative (i.e. $ K_{2}, K_{3} \ge 0 $) cases and those in which $ K_{1}, K_{4} $ are bounded away from zero or not and other restrictions on the exponents of the left hand sides are imposed.

More precisely, in order to deal with the singular-degenerate system (\ref{1}), in
Section~\ref{regularizedproblem} we introduce an approximating
system (\ref{2}) of nondegenerate-singular equations depending on a
small parameter $\varepsilon>0$. Such equations satisfy structure
conditions which, for any $\varepsilon>0$, allow the use of
well-known regularity results, i.e. H\"{o}lder continuity, from e.g.
\cite{i2}, \cite{i3}; we will use this regularity to show that the
map which associates to any couple of functions $(f,g) \in
L^\infty(Q_T) \times L^\infty(Q_T)$ the solution of the regularized
system is a compact map from $L^\infty(Q_T)\times L^\infty(Q_T)\to
L^\infty(Q_T)\times L^\infty(Q_T)$, see Lemma~\ref{compact}. Then,
for $\varepsilon>0$, the problem of showing the existence of a
non-negative solution $(u_\varepsilon,v_\varepsilon)$ to (\ref{2})
is equivalent to showing the existence of a non-negative fixed point
of such a solution map. The way we do this in
Proposition~\ref{R-moser} is based on the classical tools of the
Leray-Schauder topological degree: first, we establish uniform (with
respect to $\varepsilon>0$) a priori bounds, in this specific case
in $L^\infty(Q_T)\times L^\infty(Q_T)$, for all possible
non-negative solutions of (\ref{2}). Then, by the homotopy
invariance of the topological degree, Proposition~\ref{R-moser}
guarantees the existence of a solution
$(u_\varepsilon,v_\varepsilon)$ of (\ref{2}) in a large ball
$B_R\subset  L^\infty(Q_T)\times L^\infty(Q_T)$. Moreover, by means
of suitable conditions on the first positive eigenvalue of the
$p$-Laplacian and on some estimates on the gradient of convenient
powers of $u_\varepsilon$ and $v_\varepsilon$ established in
Lemma~\ref{stima-gradiente}, we are able to prove that
$\|u_\varepsilon\|_{L^\infty(Q_T)}$ and
$\|v_\varepsilon\|_{L^\infty(Q_T)}$ are bounded away from zero
uniformly for $\varepsilon>0$  small enough, see
Proposition~\ref{r}. To conclude, by using the uniform bounds of
$(u_\varepsilon,v_\varepsilon)$ in $L^\infty(Q_T)\times
L^\infty(Q_T)$ and the consequent uniform H\"{o}lder continuity  of
$(u_\varepsilon,v_\varepsilon)$ in $\overline{Q_T}$, we can pass to
the limit as $\varepsilon\to 0$ and show in
Theorem~\ref{thm:generale} that $(u_\varepsilon,v_\varepsilon)$, by
passing to a subsequence if necessary, converges to a solution
$(u,v)$ of (\ref{1}) with $u\not =0$ and $v\not=0$.

In Section~\ref{sec:aprioribounds}, we give conditions on the
kernels $K_i, i=1,2,3,4$, of the nonlocal terms that suffice for the
existence of uniform a priori bounds in $L^2(Q_T)\times L^2(Q_T)$
for the solutions $(u_\varepsilon, v_\varepsilon)$ of (\ref{2}). By
Lemma~\ref{limitatezza} these a priori bounds imply uniform a priori
bounds of  $(u_\varepsilon, v_\varepsilon)$ in $L^\infty(Q_T)\times
L^\infty(Q_T)$ and so, from now on, we can proceed as outlined in
Section~\ref{regularizedproblem} in order to apply
Theorem~\ref{thm:generale}. In terms of the biological
interpretations, system (\ref{1}) is a model of the interactions of
two biological species, with density $u$ and $v$ respectively,
disliking  crowding, i.e. $m,n>1$, see \cite{GmC1}, \cite{GmC} and
\cite{o}, and whose diffusion involves, as in \cite{amn}, \cite{ar},
\cite{os} and \cite{oss}, the singular $p$-Laplacian, i.e.  $1<p<2$.
The nonlocal terms $\int_\Omega K_i(\xi,t) u^2(\xi, t-\tau_i)d\xi$
and $\int_\Omega K_i(\xi,t) v^2(\xi, t-\tau_i)d\xi$ evaluate a
weighted fraction of individuals that actually interact at time
$t>0$. Nonlocal terms in biological models were first introduced in
\cite{cp} and \cite{cps}. The delayed densities $u,v$ at time $t
-\tau_i$, that appear in the nonlocal terms, take into account the
time needed to an individual to become adult, and, thus to interact
and to compete. The conditions on $K_i, i=1,2,3,4$, have the meaning
of competitive systems if $K_i\le 0, i=2,3,$ or of cooperative
systems if $K_i\ge 0, i=2,3;$ on the other hand, we always assume
that $K_i\ge 0, i=1,4,$ to take into account the intra species
competition. The term on the right hand side of each equation in
\eqref{1} denotes the actual increasing rate of the population at
$(x,t)\in Q_T$. Related results are presented in the coexistence
Theorem~\ref{thm:coercivo}, which considers the coercive case, i.e.
$K_i\ge \underline {k_i}>0, i=1,4,$ and its consequences:
Corollary~\ref{thm:coercivocooperativo} and
\ref{thm:coercivocompetitivo} for the coercive-cooperative and
coercive-competitive cases respectively. In the non-coercive case we
prove Theorem~\ref{thm:noncoerciveweakcompetitive} for competitive
systems when the diffusion is slow or normal for both the equations,
and Theorems~\ref{thm:noncoercivebruttecostanti} and
\ref{normaldiff} under a stronger assumption on $m,p,n,q$, but
without any conditions on the sign of $K_2$ and $K_3$. Observe that
these results concerning the slow and normal diffusion are relevant
for the considered biological model, in fact the slow and normal
diffusion are more realistic for the biological models as pointed
out in \cite{GmC}, \cite{i2}, \cite{o}, \cite{oss}, \cite{wyw}.
Finally, in Section~\ref{generalizzazione}, for a
generalization of system \eqref{1} which consists in having any
power $\alpha\ge 1$ of $u$ and $v$ in the nonlocal terms, we obtain,
only in the competitive case, the coexistence
Theorem~\ref{thm:generale1} and the related
Theorem~\ref{thm:coercivo1} for the coercive case and
Theorem~\ref{thm:noncoerciveweakcompetitive1} for the non-coercive
case. Note that such a generalization of system \eqref{1} is a
completely new contribution with respect to \cite{fnp} and
\cite{fnp1}.

\section{The  approximating  problem}\label{regularizedproblem}
Throughout the paper we will make the following assumptions:
\begin{Assumptions}\label{ipotesi}
\begin{enumerate}
\item The exponents $p, q, m, n$  are such that $p, q \in (1,2)$, $m
>p$ and $n > q$.
\item The delays $\tau_i\in(0,+\infty)$, $i=1,2,3,4$.
\item The functions $a$,$ b$ and $K_i$, $i=1,2,3,4$, belong to $L^\infty(Q_T)$ and are
extended to $\Omega\times\R$ by $T$-periodicity. Moreover, $a, b$
and $K_i$, $i=1,4$, are non-negative functions and there are
constants $\uk_i,\ok_i \ge 0$, $i=2,3$, such that
\[
 -\uk_i\le
K_i(x,t)\le\ok_i \text{ for }i=2,3,
\]
for a.a. $(x,t)\in Q_T$.
\end{enumerate}
\end{Assumptions}

We now recall the definition of weak solution to
\eqref{1}.
\begin{Definition}
A pair of functions $(u, v)$ is said to be a weak solution of
\eqref{1} if $u, v\,\in C(\overline{Q}_T)$, $u^m \in L^p\big((0,T);
W^{1,p}_0(\Omega)\big)$, $v^n\in L^q\big((0,T);
W^{1,q}_0(\Omega)\big)$ and $(u,v)$ satisfies
\[
\begin{aligned}
\iint_{Q_T} &\left(- u \frac{\partial \varphi}{\partial t} +|\nabla
u^m|^{p-2}\nabla u^m \nabla \varphi - a u^{p-1} \varphi +
u^{p-1}\varphi
\int_{\Omega}[K_1(\xi,t)u^2 (\xi, t- \tau_1) -\right.\\
&\left.\;\;- K_2(\xi,t)v^2(\xi, t-\tau_2)]d\xi \right)dxdt =0
\end{aligned}
\]
and
\[
\begin{aligned}
\iint_{Q_T} &\left(- v \frac{\partial \varphi}{\partial t}+|\nabla
v^n|^{q-2}\nabla v^n \nabla \varphi - b v^{q-1} \varphi +
v^{q-1}\varphi\int_{\Omega}[-K_3(\xi,t)u^2(\xi,
t-\tau_3)+\right.\\
&\left.\;\;+ K_4(\xi,t)v^2(\xi, t-\tau_4)]d\xi \right)dxdt =0,
\end{aligned}
\]
for any $\varphi\in C^1(\overline{Q}_T)$ such that $\varphi
(x,T)=\varphi(x,0)$ for any $x\in\Omega$ and $\varphi(x,t)=0$ for
any $(x,t)\in
\partial\Omega \times [0,T]$.
\end{Definition}
Here and in the following we assume that the functions $t\mapsto
u(\cdot,t)$ and $t\mapsto v(\cdot,t)$ are extended from $[0,T]$ to
$\R$ by $T$-periodicity so that $(u,v)$ is a solution defined for
all $t\in \R^+$.

In order to study system (\ref{1}) we now consider the following
nondegenerate-singular approximating $p,q$-Laplacian system
\begin{equation}\label{2}
\begin{cases}
\begin{aligned}
&\dfrac{\partial u}{\partial t}- \div((\varepsilon +
m^{p-1}u^{(m-1)(p-1)}) |\nabla u|^{p-2}\nabla u)
= \left(
a(x,t) - \int_{\Omega}K_1(\xi,t)u^2(\xi, t-\tau_1)d\xi +\right.
\\&
\left. \qquad \qquad \qquad \qquad \qquad \qquad \qquad \qquad\qquad
\qquad \qquad \qquad +\int_{\Omega}K_2(\xi,t)v^2(\xi, t-\tau_2)d\xi
\right) u^{p-1}, \quad \text{in } Q_T,\\
&\dfrac{\partial v}{\partial t}-
\div((\varepsilon + n^{q-1}v^{(n-1)(q-1)}) |\nabla v|^{q-2}\nabla v)
= \left( b(x,t) + \int_{\Omega}K_3(\xi,t)u^2(\xi, t-\tau_3)d\xi
-\right.
\\&
\left. \qquad \qquad \qquad \qquad \qquad \qquad \qquad \qquad
\qquad \qquad \qquad \qquad - \int_{\Omega}K_4(\xi,t)v^2(\xi,
t-\tau_4)d\xi\right) v^{q-1},  \quad \text{in } Q_T,
\end{aligned}\\
u(\cdot,t)|_{\partial \Omega}= v(\cdot,t)|_{\partial \Omega}=0, \quad\text{for a.a.}\; t \in (0,T),\\
u (\cdot,0)=u(\cdot,T)\text{ and }v(\cdot,0)=v(\cdot,T),
\end{cases}
\end{equation}
where $\varepsilon>0$. A solution $(u,v)$ of \eqref{1} will be then
obtained as the limit, for $\varepsilon \to 0,$ of the solutions
$(\eu  ,\ev )$ of \eqref{2} with $\eu, \ev\geq0$. For this we give
the following definition:
\begin{Definition}\label{soldebole}
A couple of functions $(\eu  , \ev   )$ is said to be a generalized (weak)
solution of \eqref{2} if \[\eu   \in L^p(0,T; W^{1, p}_0 (\Omega))
\cap C(\overline{Q}_T),\, \ev    \in L^q(0,T; W^{1, q}_0 (\Omega))
\cap C(\overline{Q}_T)\] and $(\eu  , \ev   )$ satisfies
\[
\begin{aligned}
\Int  &\left(- u   \frac{\partial \varphi}{\partial t}
+\varepsilon |\nabla u |^{p-2}\nabla u   \nabla \varphi +|\nabla u^m
|^{p-2}\nabla
u^m\nabla \varphi - a \eu ^{p-1}\varphi   \right.\\
&\left.\quad + u^{p-1}\varphi  \int_{\Omega}[K_1(\xi,t) u ^2
(\xi, t- \tau_1) - K_2(\xi,t)  v ^2(\xi, t-\tau_2)]d\xi
\right)dxdt =0
\end{aligned}
\]
and
\[
\begin{aligned}
\Int  &\left(- v    \frac{\partial \varphi}{\partial t}
+\varepsilon |\nabla v |^{q-2}\nabla v   \nabla \varphi +|\nabla v^n
|^{q-2}\nabla
v^n\nabla \varphi  - b v^{q-1}\varphi     \right.\\
&\left.\quad +  v^{q-1} \varphi  \int_{\Omega}[K_4(\xi,t) v ^2
(\xi, t- \tau_4) - K_3(\xi,t) u^2(\xi, t-\tau_3)]d\xi \right)dxdt
=0
\end{aligned}
\]
for any $\varphi \in C^1(\overline{Q}_T)$ such that $\varphi (x,T)=
\varphi(x,0)$ for any $x\in\Omega$ and $\varphi(x,t)=0$ for any
$(x,t)\in \partial\Omega \times [0,T]$.
\end{Definition}

To deal with the existence of $T$-periodic solutions $(\eu  , \ev )$
of system \eqref{2}, with $\eu  , \ev   \ge 0$ in $Q_T$,  we
introduce, for any $\varepsilon >0,$ the map $G_{\varepsilon }:
[0,1]\times L^{\infty}(Q_T) \times L^{\infty}(Q_T) \rightarrow X$,
where $X:=L^p(0,T; W^{1,p}_0(\Omega)\cap L^2(\Omega)) \times
L^q(0,T; W^{1,q}_0(\Omega)\cap L^2(\Omega))$, as follows:
\[(\sigma, f,g) \mapsto (\eu  , \ev   )=G_{\varepsilon }( \sigma, f, g)\] if and only if $(\eu  ,\ev   )$ solves
the following uncoupled problem
\begin{equation}\label{3}
\begin{cases}
\dfrac{\partial u}{\partial t}- \varepsilon \div(|\nabla  u
|^{p-2}\nabla  u) -\div(|\nabla (\sigma u^m) |^{p-2}\nabla
(\sigma u^m)) = f, & \text {in}\quad  Q_T,\vspace{4pt}\\
\dfrac{\partial v}{\partial t}- \varepsilon  \div(|\nabla  v
|^{q-2}\nabla  v) -\div(|\nabla(\sigma v^n) |^{q-2}\nabla
(\sigma v^n))= g, & \text{in}\quad Q_T,\vspace{4pt}\\
u(\cdot,t)|_{\partial \Omega}= v(\cdot,t)|_{\partial \Omega}=0,  &\text{for a.a. }t\in (0, T),\\
u (\cdot,0)=u(\cdot,T)\text{ and }v(\cdot,0)= v(\cdot,T).
\end{cases}
\end{equation}
For any fixed $\sigma \in [0,1]$ the operator $A: {\cal X}=L^p(0,T; W^{1,p}_0(\Omega)\cap L^2(\Omega))  \rightarrow {\cal X}'$,
$u \mapsto \varepsilon \div(|\nabla u |^{p-2}\nabla u) +\div(|\nabla
(\sigma u^m) |^{p-2}\nabla (\sigma u^m))$, is hemicontinuous,
strictly monotone (and hence pseudomonotone), coercive and bounded.
Thus, by \cite[Theorem 32.D]{z}, the map $G_{\varepsilon}$ is well
defined. Now, consider
\[
f(\alpha,\beta):= \left(a-\int_{\Omega}K_1(\xi,\cdot)\,\alpha^2(\xi,
\cdot-\tau_1)d\xi+\int_{\Omega}K_2(\xi,\cdot)\,\beta^2(\xi,
\cdot-\tau_2)d\xi\right)\alpha^{p-1}
\]
and
\[
g(\alpha,\beta):=\left(b+\int_{\Omega}K_3(\xi,\cdot)\,\alpha^2(\xi,
\cdot-\tau_3)d\xi-\int_{\Omega}K_4(\xi,\cdot)\,\beta^2(\xi,
\cdot-\tau_4)d\xi\right)\beta^{q-1},
\]
where $\alpha \text{ and } \beta$ belong to  $L^{\infty}(Q_T)$.
Clearly, if the nonnegative functions $\eu  ,\ev   \in
L^{\infty}(Q_T)$ are such that $(\eu  ,\ev   )=G_{\varepsilon
}\big(1, f(\eu  ,\ev   ),g(\eu  ,\ev   )\big)$, then $(\eu ,\ev )$
is also a solution of \eqref{2} (with $\eu  $ and $\ev   \ge0$) in
$Q_T$. Hence, the existence of a nonnegative solution of \eqref{2}
is equivalent to the existence of a fixed point $(\alpha, \beta)$ of
the map $(\alpha, \beta)\mapsto G_{\varepsilon }\big( 1, f(\alpha,
\beta),g(\alpha, \beta)\big)$ with $\alpha$ and $\beta \ge0$.

Let $T_{\varepsilon}(\sigma, \alpha, \beta) := G_{\varepsilon
}\big(\sigma, f(\alpha, \beta),g(\alpha, \beta)\big).$
By \cite[Theorem 1.1]{i2} and \cite[Theorem 1.3]{i3} we have the
H\"{o}lder estimate in the interior of $Q_T$ of the solutions to
(\ref{3}). Moreover, the property of positive geometric density of
$\partial \Omega$, the fact that the Dirichlet data is H\"{o}lder
continuous and the periodicity condition ensure that one can obtain
the H\"{o}lder estimate up to the parabolic boundary of $Q_T$, see
the comments to \cite[Theorem 1.1]{i2} and \cite[Theorem 7.1]{i3}.

We will need the following result, which was proved in \cite[Lemma
2.1]{fnp1} for the $p,q=2$, $m,n>1$, but whose proof is the same, so
we omit it.
\begin{Lemma}\label{compact}
Let $ (\alpha, \beta) \in L^{\infty}(Q_T) \times L^{\infty}(Q_T)$
and let $\varepsilon >0$. Then $T_{\varepsilon}:[0,1]\times
L^{\infty}(Q_T)\times L^{\infty}(Q_T)\rightarrow
L^{\infty}(Q_T)\times L^{\infty}(Q_T)$, $(\sigma, \alpha,
\beta)\mapsto T_{\varepsilon}(\sigma, \alpha, \beta)=(\eu ,\ev )$ is
compact. Moreover  $\eu  , \ev   \in C(\overline {Q}_T)$.
\end{Lemma}

Our aim is to prove the existence of $T$-periodic solutions $\eu  ,
\ev    \in C(\overline{Q}_T)$ of problem \eqref{2} with  $\eu  ,$
$\ev $ $>0$ in $Q_T$, for all $\varepsilon>0$ small enough, as
positive fixed points of the map $(\alpha, \beta)\mapsto
T_{\varepsilon}(1, \alpha, \beta)$. As a first step we prove the
following result.

\begin{Proposition}\label{positivity}
If the non-trivial pair $(\eu ,\ev )$ solves
\begin{equation}\label{xcorrezione}
(u,v)=G_{\varepsilon} \big(\sigma,  f(u^+,v^+)+(1-\sigma),
g(u^+,v^+)+(1-\sigma)\big)
\end{equation}
for some $\sigma\in[0,1]$, then
\[
\eu  (x,t)\ge 0 \text{ and }\ev   (x,t)\ge 0\quad\mbox{for any
}(x,t)\in Q_T.
\]
Moreover, if $\eu   \neq 0$ or $\ev    \neq 0$, then $\eu  >0$ or
$\ev   >0$ in $Q_T$, respectively.
\end{Proposition}
\begin{proof}
Assume that $(\eu  , \ev   )$ solves \eqref{xcorrezione} with $\eu
\neq 0$ for some $\sigma\in[0,1]$. We first prove that $\eu   \ge0$.
Multiplying the first equation of \eqref{3}, where $f(\alpha,
\beta)$ is replaced by $ f(\eu  ^+, \ev ^+)+ (1- \sigma)$, by $\eu
^-:=\mbox{min}\{0, \eu  \}$, integrating on $Q_T$ and passing to the
limit in the Steklov averages $(\eu  )_h\in H^1(Q_{T- \delta})$,
$\delta,h>0$, in a standard way \cite[p.85]{lsu}, we obtain
\[
 \varepsilon \Int |\nabla \eu|^{p-2} \nabla \eu\nabla \eu ^-dxdt + \Int|\nabla(\sigma \eu^m)|^{p-2}\nabla
(\sigma \eu^m)\nabla \eu ^- dxdt=\Int  (1- \sigma)\eu ^-dxdt \le 0,
\]
by the $T$-periodicity of $\eu  $ and taking into account that $\eu
^+\,\eu  ^-=0$. Hence we obtain
\[
\varepsilon\Int |\nabla \eu  ^-|^pdxdt \le\varepsilon \Int |\nabla
\eu|^{p-2} \nabla u\nabla \eu ^- dxdt+ \Int|\nabla(\sigma
\eu^m)|^{p-2}\nabla (\sigma \eu^m)\nabla \eu ^- dxdt \le 0.
\]
Thus
\[
\Int |\nabla \eu  ^-|^p dxdt= 0.
\]
The Poincar\'{e} inequality gives
\[
0 \le \int_{\Omega}|\eu  ^-|^p dx \le \frac{1}{\mu_p}
\int_{\Omega}|\nabla \eu ^-|^p dx \quad \mbox{for all }t,
\]
where $\mu_p$ is the first positive eigenvalue of the problem
\[
\begin{cases}
- \diver(|\nabla z|^{p-2}\nabla z) = \mu |z|^{p-2}z, & x \in \Omega,\\
z =0, &x \in \partial \Omega,
\end{cases}
\]
(see, for example, \cite{kl}). Integrating over $(0,T)$, we have
\[
0 \le \Int |\eu  ^-|^p dxdt \le \frac{1}{\mu_p} \Int |\nabla \eu
^-|^p dxdt = 0,
\]
which, together with the boundary conditions and the fact that $\eu
^- \in C(\overline{Q}_T)$, implies $\eu  ^-(x,t)=0$ for all $(x,t)
\in Q_T$.

Now we prove that $\eu >0$ in $Q_T$. Since $\eu$ is non-trivial,
there exists $(x_0,t_0)\in \Omega \times (0,T]$ such that $\eu
(x_0,t_0) > 0$. Hence  $\eu(x,t)>0$ for all $(x,t) \in Q_T$ (see
\cite[p.3]{BD} and \cite{K}).  In the same way, one can prove that
$\ev \neq0$ implies $\ev   (x,t) >0$ for all $(x,t) \in Q_T$.
\end{proof}

The next lemma is crucial to prove Proposition \ref{R-moser} below.
\begin{Lemma}\label{limitatezza}
Let $K >0$ and assume that $u$ is a non-negative $T-$periodic
continuous function such that $x \mapsto u(x,t) \in
W_0^{1,p}(\Omega)$ for all $t\in[0,T]$ and which satisfies
\[
u_t-\varepsilon{\rm div}(|\nabla u |^{p-2}\nabla u)- {\rm
div}(|\nabla u^m |^{p-2}\nabla u^m) \le Ku^{p-1}\quad \text{in}
\quad Q_T.
\]
Then there exists $R >0$ such that
\[
\|u\|_{L^\infty(Q_T)}\le R \mbox{ for all $\varepsilon>0$}.
\]
\end{Lemma}

\begin{proof} We follow Moser's technique to show the stated a priori bounds.
Multiplying
\[
u_t-\varepsilon\div(|\nabla u |^{p-2}\nabla u)- \div(|\nabla u^m
|^{p-2}\nabla u^m) \le Ku^{p-1}
\]
by $u^{s+1}$, with $s\ge 0$, integrating over $\Omega$ and passing
to the limit as $h\to 0$ in the Steklov averages $u_h\in H^1(Q_{T-
\delta})$, $\delta,h>0$, we have
\[
\frac{1}{s+2}\frac{d}{dt}\|u(t)\|_{L^{s+2}(\Omega)}^{s+2}+\int_{\Omega}(\varepsilon
|\nabla u|^{p-2} \nabla u +  |\nabla u^m|^{p-2} \nabla u^m )\nabla
u^{s+1}dx \le K\|u(t)\|_{L^{s+p}(\Omega)}^{s+p}\le C_{|\Omega|}
\|u(t)\|_{L^{s+2}(\Omega)}^{s+p},
\]
and thus
\[
\frac{d}{dt}\|u(t)\|_{L^{s+2}(\Omega)}^{s+2}
+(s+1)(s+2)m^{p-1}\int_{\Omega}u^{(m-1)(p-1)+s}|\nabla u|^pdx \le
(s+2)C_{|\Omega|} \|u(t)\|_{L^{s+2}(\Omega)}^{s+p}
\]
where $C_{|\Omega|}:=\sup_{s\ge 0}K |\Omega|^{1-\frac{s+p}{s+2}}$.
Since $m,p > 1$, it follows
\begin{equation}\label{lim2}
\begin{aligned}
&\frac{d}{dt}\|u(t)\|_{L^{s+2}(\Omega)}^{s+2} +
\frac{s+2}{[m(p-1)+s+1]^p} \int_\Omega
\left|\nabla u^{\frac{m(p-1)+s+1}{p}}\right|^p dx \\
&\le\frac{d}{dt}\|u(t)\|_{L^{s+2}(\Omega)}^{s+2} +
(s+2)m^{p-1}\left(\frac{p}{m(p-1)+s+1}\right)^p \int_\Omega
\left|\nabla u^{\frac{m(p-1)+s+1}{p}}\right|^p dx \\
&\le \frac{d}{dt}\|u(t)\|_{L^{s+2}(\Omega)}^{s+2} +
(s+1)(s+2)m^{p-1}\int_{\Omega}u^{(m-1)(p-1)+s}|\nabla u|^pdx\\
&\le C_{|\Omega|}(s+2)\|u(t)\|_{L^{s+2}(\Omega)}^{s+p}.
\end{aligned}
\end{equation}
For $\varepsilon$ fixed and $k=1,2,\dots$, setting
\[
s_k:=2p^k +\frac{p^{k}-p}{p-1}+m-1,\quad
\alpha_k:=\frac{p(s_k+2)}{m(p-1)+ s_k +1},\quad
w_k:=u^{\frac{m(p-1)+s_k+1}{p}},
\]
we obtain by \eqref{lim2}
\begin{equation}\label{lim3}
\frac{d}{dt}\|w_k(t)\|_{{L^{\alpha_k}}(\Omega)}^{\alpha_k}
+\frac{(s_k+2)}{[m(p-1)+s_k+1]^p}\|\nabla w_k(t)\|^p_{L^p(\Omega)}
\le
C_{|\Omega|}(s_k+2)\|w_k(t)\|_{{L^{\alpha_k}}(\Omega)}^{\alpha_k\frac{s_k+p}{s_k+2}}.
\end{equation}
Now, let us fix $s>p$ such that the continuous Sobolev immersion
$W_0^{1,p}(\Omega)\subset L^s(\Omega)$ holds and observe that since
$s_k \rightarrow +\infty$, as $k \rightarrow +\infty$, there exists
$k_0$ such that $\alpha_k \in (1, s)$ for all $k \ge k_0$. By the
interpolation and the Sobolev inequalities, it results
\[
\|w_{k}(t)\|_{L^{\alpha_{k}}(\Omega)} \le
\|w_k(t)\|_{L^1(\Omega)}^{\theta_k}\| w_k(t)\|_{L^s(\Omega)}^{1-
\theta_k}\le C \|w_k(t)\|_{L^1(\Omega)}^{ \theta_k}\|\nabla
w_k(t)\|_{L^p(\Omega)}^{1- \theta_k}
\]
for all $k \ge k_0$. Here $\theta_k=(s-\alpha_k)/[\alpha_k(s-1)]$
and $C$ is a positive constant. Using the fact that
$\|w_k(t)\|_{L^1(\Omega)}=\|w_{k-1}(t)\|_{L^{\alpha_{k-1}}(\Omega)}^{\alpha_{k-1}}$
and defining
$x_{k}:=\sup_{t\in\R}\|w_{k}(t)\|_{L^{\alpha_{k}}(\Omega)}$, one has
\[
\begin{aligned}
\|w_{k}(t)\|_{L^{\alpha_{k}}(\Omega)}^{\frac{p}{1- \theta_k}}\le
&C\|w_{k-1}(t)\|_{L^{\alpha_{k-1}}(\Omega)}^{p\alpha_{k-1}\frac{\theta_k}{1-\theta_k}}
\|\nabla w_k(t)\|_{L^p(\Omega)}^p\\
\le &Cx_{{k-1}}^{p\alpha_{k-1}\frac{\theta_k}{1- \theta_k}} \|\nabla
w_k(t)\|_{L^p(\Omega)}^p
\end{aligned}
\]
for all $k \ge k_0$. Thus, by \eqref{lim3},
\begin{equation}\label{lim4}
\begin{aligned}
\frac{d}{dt}\|w_k(t)\|_{{L^{\alpha_k}}(\Omega)}^{\alpha_k} &\le
C_{|\Omega|}(s_k+2)\|w_k(t)\|_{{L^{\alpha_k}}(\Omega)}^{\alpha_k\frac{s_k+p}{s_k+2}}\\
&-\frac{s_k+2}{C[m(p-1)+s_k+1]^p}
\|w_k(t)\|_{L^{\alpha_k}(\Omega)}^{\frac{p}{1- \theta_k}}\,
 x_{{k-1}}^{p\alpha_{k-1}\frac{\theta_k}{\theta_k-1}}\\
=&\left(C_{|\Omega|}-\frac{1}{C[m(p-1)+s_k+1]^p}\|w_k(t)\|_{L^{\alpha_k}(\Omega)}^{\frac{p}{1-\theta_k}-\alpha_k\frac{s_k+p}{s_k+2}}
 x_{{k-1}}^{p\alpha_{k-1}\frac{\theta_k}{\theta_k-1}}\right)(s_k+2)\|w_k(t)\|_{{L^{\alpha_k}}(\Omega)}^{\alpha_k\frac{s_k+p}{s_k+2}}
\end{aligned}
\end{equation}
for all $k \ge k_0$. By Lemma \ref{appendice} and \eqref{lim4}, one
has
\begin{equation}\label{lim5}
\|w_k(t)\|_{L^{\alpha_k}(\Omega)}\le
\left(C_{|\Omega|}M_kx_{k-1}^{p\alpha_{k-1}\frac{\theta_k}{1-\theta_k}}\right)^{\eta_k}
\end{equation}
for all $k \ge k_0$, where
$\eta_k:=\frac{(1-\theta_k)(s_k+2)}{p(s_k+2)-\alpha_k(s_k+p)(1-\theta_k)}$
and $M_k:=C[m(p-1)+s_k+1]^p$. By definition of $x_k$ and
\eqref{lim5}, we get
\[
x_k \le \left(C_{|\Omega|}M_k\right)^{\eta_k} x_{k-1}^{\nu_k}
\]
for all $k \ge k_0$, with
$\nu_k:=p\alpha_{k-1}\theta_k(s_k+2)/[p(s_k+2)-\alpha_k (s_k+p)
(1-\theta_k)]$.

If $x_{k_n} \le 1$ along a sequence $k_n\to+\infty$, then one has
$\|u\|_{L^\infty(Q_T)}\le 1$ by the very definition of $x_k$ and the
lemma is proved. On the other hand, assume $x_{k}>1$ for all $k \ge
k_0$ and observe that there exists $\overline {k}_0$ such that, for
all $k\ge\overline{k}_0$,
\[
\eta_k\le 1/(p\theta) \qquad \text{and} \qquad \nu_k\le
\alpha_{k-1}.
\]
Here $\theta:=(s-p)/[p(s-1)]$. Without loss of generality, assume
$k_0=\max\{\overline{k}_0,k_0\}$. Then, there exists a positive
constant $A$ such that
\[
\begin{aligned}
x_k\le &\left(C_{|\Omega|}C\right)^{\eta_k}[m(p-1)+s_k+1]^{p\eta_k}x_{k-1}^{\nu_k}\\
\le &\left(C_{|\Omega|}C\right)^{\eta_k}
 \left(mp+\frac{2p^{k+1}}{p-1}\right)^{p\eta_k}x_{k-1}^{\nu_k}
\le Ap^{\frac{k+1}{\theta}}x_{k-1}^{\alpha_{k-1}}
\end{aligned}
\]
for all $k \ge k_0$. Thus
\begin{equation}\label{estimate}
\begin{aligned}
\log x_k \le &\log A+\frac{k+1}{\theta}\log p + \alpha_{k-1} \log x_{k-1} \\
\le &\left(1+\sum_{i=1}^{k-k_0-1}
\prod_{j=1}^{i}\alpha_{k-j}\right)\log A +
\left(k+1+\sum_{i=k_0+2}^{k}i\prod_{j=1}^{k+1-i}\alpha_{k-j}\right)\frac{\log
p}{\theta} + \left(\prod_{j=1}^{k-k_0}\alpha_{k-j}\right)\log
x_{k_0}.
\end{aligned}
\end{equation}
 Since
\[
\alpha_k = 1 + \dfrac{1 - m (p - 1)}{m (p - 1) +
s_k +1} \le 1 + \dfrac{|1 - m (p - 1)|}{2 p^k}\,,
\]
for $i \le k$ we have that\[
\begin{aligned}
\log \left(\dfrac{1}{p^i} \prod_{j=1}^i \alpha_{k - j}\right) & = \sum_{j=1}^{i} \log \dfrac{\alpha_{k - j}}{p} \\
& \le \sum_{j=1}^{i} \log \left(1 + \dfrac{|1 - m (p - 1)|}{2 p^{k - j + 1}}\right) \\
& \le \dfrac{|1 - m (p - 1)|}{2 p^k} \sum_{j=1}^i p^{j - 1} \\
& \le \dfrac{|1 - m (p - 1)|}{2 (p - 1)}\,,
\end{aligned}
\]
which means that
\[
\prod_{j=1}^i \alpha_{k - j} \le M p^i \qquad \text{with }M = \exp
\dfrac{|1 - m (p - 1)|}{2 (p - 1)}\,.
\]
 Therefore, from \eqref{estimate} we obtain
\begin{equation}\label{estimatebis}
\begin{aligned}
\dfrac{\log x_k}{M} & \le \log A \sum_{i=0}^{k - k_0 - 1} p^i + \dfrac{\log p}{\theta} \sum_{i=k_0 +2}^{k + 1} i p^{k + 1 - i} + p^{k - k_0} \log x_{k_0} \\
&\le \frac{\log p}{\theta}\frac{p^{k-k_0}}{(p-1)^2}
 \left[k_0(p-1)+2p-1\right]+\log A \frac{1-p^{k-k_0}}{1-p}+ p^{k-k_0}\log x_{k_0}.
\end{aligned}
\end{equation}
In fact, taking $x= \frac{1}{p}$ in $\displaystyle x\frac{d}{dx}\sum
_{i=0}^{k+1} x^i = x  \frac{d}{dx}\left( \frac{1-
x^{k+2}}{1-x}\right), $ it results
\[
\begin{aligned}
\sum_{i=k_0+2}^{k+1} ip^{k+1-i}  &= \frac{p^{k+3}}{(p-1)^2}
\left[\frac{1}{p^{k+2}} \left(\frac{k+1}{p}-k-2\right) -
\frac{1}{p^{k_0+2}} \left(\frac{k_0+1}{p}-k_0-2\right)\right]
\\&\le\frac{p^{k+3}}{(p-1)^2}\frac{1}{p^{k_0+2}}
\left(k_0+2-\frac{k_0+1}{p}\right) =\frac{p^{k-k_0}}{(p-1)^2}
\left[k_0(p-1)+2p-1\right].
\end{aligned}
\]
Then, by \eqref{estimatebis}, it follows
\[
x_k^{1/M} \le A^{\frac{1-p^{k-k_0}}{1-p}}
p^{\frac{p^{k-k_0}}{\theta(p-1)^2}
 \left[k_0(p-1)+2p-1\right]}
x_{k_0}^{p^{k-k_0}}.
\]
Since
$x_{k}=\sup_{t\in\R}\|u(t)\|_{s_{k}+2}^{\frac{m(p-1)+s_{k}+1}{p}}$,
we obtain
\[
\begin{aligned}
\|u(t)\|_{L^\infty(\Omega)} \le
&\lim_{k \rightarrow\infty}\|u(t)\|_{s_k + 2}\\
\le &\limsup_{k
\rightarrow\infty}\left\{A^{\frac{p}{m(p-1)+s_{k}+1}\frac{1-p^{k-k_0}}{1-p}}
x_{k_0}^{\frac{p^{k-k_0+1}}{m(p-1)+s_{k}+1}}
p^{\frac{p^{k-k_0+1}(k_0(p-1)+2p-1)}
{\theta(p-1)^2(m(p-1)+s_{k}+1)}
 }
\right\}^M\\&=:R \quad \forall \; t \in \R,
\end{aligned}
\]
where $R$ is a positive constant. Hence $ \sup_{t \in \R}
\|u(t)\|_{L^\infty(\Omega)} \le R. $ It remains to prove that $R$ is
independent of $\varepsilon$ as claimed. To this aim it is
sufficient to prove that there exists $C>0$, independent of
$\varepsilon$, such that $x_{k_0} \le C$. In fact, by \eqref{lim3}
with $s_0:= s_{k_0}$, it follows
\begin{equation}\label{lim6}
\frac{d}{dt}\|u(t)\|_{{L^{s_0+2}}(\Omega)}^{s_0+2}
\!+\frac{s_0+2}{[m(p-1)+s_0+1]^p}\!\!\int_\Omega \!\left|\nabla
u^{\frac{m(p-1)+s_0+1}{p}}\right|^p dx\!\!\le\!
C_{|\Omega|}(s_0+2)\|u(t)\|_{{L^{s_0+2}}(\Omega)}^{s_0+p}\!.
\end{equation}
 Moreover, we have
\[
\|u(t)\|_{{L^{s_0+2}}(\Omega)}^{m(p-1)+s_0+1} \le C_1 \left[
\int_{\Omega}
\left(u^{\frac{m(p-1)+s_0+1}{p}}\right)^{p\frac{s_0+2}{s_0+p}}dx\right]^{\frac{s_0+p}{s_0+2}}
\]
by H\"{o}lder's inequality  with $r=\frac{m(p-1)+s_0+1}{s_0+p}$.
Now, without loss of generality, we can assume that $k_0$ is chosen
large enough such that the continuous Sobolev immersion
$W^{1,p}_0(\Omega) \subset L^{p\frac{s_0+2}{s_0+p}}(\Omega)$ holds
and, hence, we deduce that
\[
\|u(t)\|_{{L^{s_0+2}}(\Omega)}^{m(p-1)+s_0+1} \le C_2 \|\nabla
u^{\frac{m(p-1)+s_0+1}{p}}\|_{{L^p}(\Omega)}^{p},
\]
for a positive constant $C_2$.  Thus, using \eqref{lim6}, one has
\[
\begin{aligned}
&\frac{d}{dt}\|u(t)\|_{{L^{s_0+2}}(\Omega)}^{s_0+2}
+\frac{s_0+2}{C_2[m(p-1)+s_0+1]^p}\|u(t)\|_{{L^{s_0+2}}(\Omega)}^{m(p-1)+s_0+1}
\\&\le  \frac{d}{dt}\|u(t)\|_{{L^{s_0+2}}(\Omega)}^{s_0+2}
+\frac{s_0+2}{[m(p-1)+s_0+1]^p}\|\nabla
u^{\frac{m(p-1)+s_0+1}{p}}\|_{{L^p}(\Omega)}^{p} \\&\le
C_{|\Omega|}(s_0+2)\|u(t)\|_{{L^{s_0+2}}(\Omega)}^{s_0+p}.
\end{aligned}
\]
Hence
\[
\frac{d}{dt}\|u(t)\|_{{L^{s_0+2}}(\Omega)}^{s_0+2}
\le\|u(t)\|_{{L^{s_0+2}}(\Omega)}^{s_0+p} \left(C_{|\Omega|}(s_0+2)
-M\|u(t)\|_{{L^{s_0+2}}(\Omega)}^{(m-1)(p-1)} \right),
\]
where $M:= \frac{s_0+2}{C_2[m(p-1)+s_0+1]^p}$. This inequality and
Lemma \ref{appendice} below imply that
\[
\|u(t)\|_{{L^{s_0+2}}(\Omega)} \le \{C_2C_{|\Omega|}[m(p-1)+s_0+1]^p\}^{\frac{1}{{(m-1)(p-1)}}} \quad \forall t \in \R.
\]
Thus, there exists $C>0$, independent of $\varepsilon$, such that
$x_{k_0}=\sup_{t\in\R}\|u(t)\|_{L^{s_{0}+2}(\Omega)}^{\frac{m(p-1)+s_{0}+1}{p}}
\le C$, and this concludes the proof.
\end{proof}
\begin{Lemma}\label{appendice}
Let $f: \R \rightarrow (0,+\infty)$ be a differentiable and $T$-periodic function; suppose that there exist positive constants
$s, \alpha, \beta, \gamma$ such that
\[
f'(t) \le f^s(t)(\beta-\gamma f^\alpha(t)),
\]
for all $t\in \R$. Then $\beta -\gamma f^\alpha(t) \ge 0$ for all $t
\in \R$.
\end{Lemma}
\begin{proof}
By the periodicity and continuity of $f$, let $ t_{0} $ be any point in which $ f $ attains its maximum.
Then we have:
\[
\beta - \gamma f^{\alpha}( t ) \ge \beta - \gamma f^{\alpha}( t_{0} ) \ge \frac{ f'( t_{0} ) }{ f^{s}( t_{0} ) } = 0 \qquad \forall t\in\R. \qedhere
\]
\end{proof}
Next, we show that the map $I-G_{\varepsilon}:\{1\}\times L^\infty(Q_T)\times L^\infty(Q_T) \rightarrow L^\infty(Q_T)\times L^\infty(Q_T)$
has the Leray-Schauder topological degree different
from zero in the intersection of a sufficiently large ball centered
at the origin with the cone of non-negative functions.

From now on we make the following assumption:
\begin{Assumptions}\label{ipotesi1} There exist two positive constants $C_1,C_2$ such that
\begin{enumerate}
\item for all $\varepsilon>0$ and all  solution pairs $(\eu  ,\ev   )$ of
\begin{equation}\label{eq:rhovarepsilon}
(u,v)=G_{\varepsilon} \big(1, f(u^+,v^+), g(u^+,v^+)\big),
\end{equation}
it results
\begin{equation}\label{eq:ipotesimaggiorazioni} \|\eu
\|_{L^2(Q_T)}^2\le C_1\text{ and }\|\ev   \|_{L^2(Q_T)}^2\le C_2,
\end{equation}
\item
$ \displaystyle \min \left\{\dfrac{1}{T}\iint_{Q_T}a(x,t)e_p^p(x)dxdt-\dfrac{\uk_2C_2}{T},
\dfrac{1}{T}\iint_{Q_T}b(x,t)e_q^q(x)dxdt-\dfrac{\uk_3C_1}{T}\right\}>0,
$ where $\uk_2, \uk_3$ are as in Hypotheses \ref{ipotesi}.3, 
and, for any $ r > 1 $, $e_r$ stands for the positive eigenvector associated to the first eigenvalue $\mu_r$ of $-\Delta_r$ with Dirichlet boundary conditions and
normalized in such a way that $\|e_r\|_{L^r(\Omega)}=1$.
\end{enumerate}
\end{Assumptions}
\begin{rem}\label{rem:onaprioriboundsandcooperation}
The last assumptions, and in particular statement 2, will grant the periodic coexistence, that is the existence of a $T$-periodic solution couple $ ( u, v ) $ with
$ u, v $ both non-negative and non-trivial (see Theorem~\ref{thm:generale}).
Generally speaking, the \emph{explicit} knowledge of the a priori bounds $ C_{1}, C_{2} $ is required to check the validity
of statement 2 of Hypotheses \ref{ipotesi1}.
This is the task we will devote ourselves to in the next Section.
A notable exception is the case in which $ K_{2}, K_{3} $ are not negative (namely, the cooperative case), since one can choose $ \uk_{2} = \uk_{3} = 0 $ and, thus,
the second part of Hypotheses \ref{ipotesi1} is readily satisfied regardless of the values of $ C_{1}, C_{2} $, if neither of the coefficients $ a, b $ is trivial.
\end{rem}
The next result shows how we can pass from an $L^2$-estimate to an
$L^\infty$-estimate.
\begin{Proposition}\label{R-moser}
There is a constant $R>0$ such that
\[
\|\eu  \|_{L^\infty(Q_T)},\|\ev   \|_{L^\infty(Q_T)}<R
\]
for all solution pairs $(\eu  ,\ev   )$ of \eqref{eq:rhovarepsilon}
with $\varepsilon >0$ sufficiently small. In particular, one has
that
\[
\deg\big((u,v)-G_{\varepsilon}\big(1, f(u^+,v^+),
g(u^+,v^+)\big),B_R,0\big)=1.
\]
\end{Proposition}
\begin{proof}
Assume $ \eu \neq 0 $, thus $ \eu > 0 $ and $ \ev \ge 0 $ in $ Q_T $ by
Proposition~\ref{positivity}. Multiplying by $ \eu $ the first
equation of \eqref{2}, integrating over $\Omega$ and using the
Steklov averages $(\eu  )_h \in H^1(Q_{T- \delta})$, $\delta,h>0$,
see \cite[p.85]{lsu}, we obtain
\[
\begin{aligned}
&\frac{1}{2} \frac{d}{dt} \int_\Omega (\eu  )_h^2 dx
+\varepsilon\int_{\Omega}|\nabla(\eu  )_h|^p dx + m^{p-1}\int_\Omega
(\eu  )_h^{(m-1)(p-1)}|\nabla (\eu  )_h|^p dx \\
&\le \left(\|a\|_{L^\infty(Q_T)} + \int_\Omega K_2(\xi,t)(\ev
)_h^2(\xi,
t-\tau_2)d\xi \right)\int_\Omega (\eu  )_h^p dx\\
&\le |\Omega|^{1-\frac{p}{2}}\left( \|a\|_{L^\infty(Q_T)} +
\int_\Omega K_2(\xi,t)(\ev  )_h^2(\xi, t-\tau_2)d\xi
\right)\left(\int_\Omega (\eu  )_h^2 dx\right)^{\frac{p}{2}}.
\end{aligned}
\]
Thus
\begin{equation}\label{eq:log}
\begin{aligned}
& \frac{\frac{1}{2} \frac{d}{dt} \int_\Omega (\eu  )_h^2 dx+
\varepsilon\int_{\Omega}|\nabla(\eu  )_h|^pdx+ m^{p-1}\int_\Omega
(\eu )_h^{(m-1)(p-1)}|\nabla (\eu)_h|^pdx}{\left(\int_\Omega (\eu
)_h^2 dx\right)^{\frac{p}{2}}}\\&\le
|\Omega|^{1-\frac{p}{2}}\left(\|a\|_{L^\infty(Q_T)}+\|K_2\|_{L^\infty(Q_T)}\int_{\Omega}(\ev
)_h^2(\xi,t-\tau_2)d\xi\right).
\end{aligned}
\end{equation}
Since $t\mapsto\|u(t)\|_{L^2(\Omega)}$ is continuous in $[0,T]$,
there exist $t_1$ and $t_2$ in $[0,T]$ such that
\[
\int_\Omega\eu ^2(x,t_1)dx=\min_{t\in[0,T]}\int_\Omega\eu
^2(x,t)dx\] and
\[
\int_\Omega\eu ^2(x,t_2)dx=\max_{t\in[0,T]}\int_\Omega\eu ^2(x,t)dx.
\]
Without loss of generality, by periodicity, we can assume that $t_1
<t_2$. Then, integrating \eqref{eq:log} between $t_1$ and $t_2$ and
passing to the limit as $h\rightarrow 0$, we find
\[
\int_{t_1}^{t_2} \left(\int_\Omega
\eu^2dx\right)^{-\frac{p}{2}}\frac{d}{dt}\left(\int_\Omega
\eu^2dx\right)dt \le 2|\Omega|^{1-\frac{p}{2}}
\int_0^T\left(\|a\|_{L^\infty(Q_T)}+\|K_2\|_{L^\infty(Q_T)}\int_{\Omega}\ev
^2(\xi,t-\tau_2)d\xi\right)dt.
\]
Thus
\[
\left(\int_\Omega\eu^2(x,t_2)dx\right)^{\frac{2-p}{2}}-\left(\int_\Omega\eu^2(x,t_1)dx\right)^{\frac{2-p}{2}}\le
C( T\|a\|_{L^\infty(Q_T)}+ \|K_2\|_{L^\infty(Q_T)}C_2),
\]
where $C:=(2-p)|\Omega|^{1-\frac{p}{2}}$. Hence
\[
\left(\int_\Omega\eu^2(x,t_2)dx\right)^{\frac{2-p}{2}} \le
\left(\int_\Omega\eu^2(x,t_1)dx\right)^{\frac{2-p}{2}} +C(
T\|a\|_{L^\infty(Q_T)}+ \|K_2\|_{L^\infty(Q_T)}C_2),
\]
or, equivalently,
\[
\max_{t \in[0,T]} \int_\Omega \eu^2(x,t)dx \le \left\{\left(\min_{t
\in [0,T]}\int_\Omega\eu^2(x,t)dx\right)^{\frac{2-p}{2}} +C(
T\|a\|_{L^\infty(Q_T)}+
\|K_2\|_{L^\infty(Q_T)}C_2)\right\}^{\frac{2}{2-p}}.
\]
This implies that there exists a constant $\gamma>0$, independent of
$\varepsilon$, such that
\[
\max_{t\in[0,T]}\int_\Omega\eu ^2(x,t)dx\le\gamma.
\]
Otherwise, for all $\gamma>0$ there would exist  $\varepsilon>0$ such that the
corresponding solution $\eu$ satisfies
\[
\gamma < \max_{t \in[0,T]} \int_\Omega \eu^2(x,t)dx \le
\left\{\left(\min_{t \in
[0,T]}\int_\Omega\eu^2(x,t)dx\right)^{\frac{2-p}{2}} +C(
T\|a\|_{L^\infty(Q_T)}+
\|K_2\|_{L^\infty(Q_T)}C_2)\right\}^{\frac{2}{2-p}}.
\]
Using the fact that $\frac{2}{2-p} >1$ and integrating the previous
inequality on $[0,T]$ for sufficiently large $\gamma$, one would
have
\[
\gamma T \le \iint_{Q_T}\eu^2(x,t)dxdt+CT( T\|a\|_{L^\infty(Q_T)}+
\|K_2\|_{L^\infty(Q_T)}C_2),
\]
that is $\eu$ is unbounded in $L^2(Q_T)$, in contradiction with
Hypothesis \ref{ipotesi1}.1. Of course, an analogous inequality holds for $\ev $.

Now, we have
\begin{equation}\label{5}
\begin{aligned}
&\dfrac{\partial \eu  }{\partial t}-\varepsilon\div(|\nabla \eu|^{p-2}\nabla \eu)- \div(|\nabla \eu^m |^{p-2}\nabla \eu)\le\\
&\le\left(\|a\|_{L^\infty(Q_T)}+\|K_2\|_{L^\infty(Q_T)}\max_{t\in
[0,T]}\int_\Omega\ev   ^2(x,t)dx\right)\eu^{p-1}   \le
\\&\le(\|a\|_{L^\infty(Q_T)}+ \|K_2\|_{L^\infty(Q_T)}\gamma)\eu^{p-1}  ,
\end{aligned}
\end{equation}
i.e.
\[
\dfrac{\partial \eu  }{\partial t}-\varepsilon\div(|\nabla
\eu|^{p-2}\nabla \eu)- \div(\nabla \eu^m|\nabla \eu^m |^{p-2})\le
K\eu^{p-1},\quad \text{in}\; Q_T,
\]
where $K:=\|a\|_{L^\infty(Q_T)}+ \|K_2\|_{L^\infty(Q_T)}\gamma$.
By Lemma \ref{limitatezza}
we conclude that $\|\eu \|_{L^\infty(Q_T)}\le R_1$
for some $R_1>0$ independent of $\varepsilon$. Analogously, $\|\ev
\|_{L^\infty(Q_T)}\le R_2$ for some constant $R_2>0$. Therefore it
is enough to choose $R>\max\{R_1, R_2\}$.

The previous calculations also show that any solution pair of
\[
(u,v)= G_{\varepsilon}(1,\rho f(u^+,v^+),\rho g(u^+,v^+))
\]
with $\rho\in [0,1]$ satisfies the same inequality (\ref{5}).
Therefore, the homotopy invariance property of the Leray-Schauder
degree implies that
\[
\begin{aligned}
&\deg\big((u,v)-T_{\varepsilon}(1,u^+,v^+),B_R,0\big) \\
= &\deg\big((u,v)-G_{\varepsilon}(1,\rho f(u^+,v^+),\rho
g(u^+,v^+)),B_R,0\big)
\end{aligned}
\]
for any $\rho\in [0,1].$ If we take $\rho=0$, using the fact that
$G_{\varepsilon}$ at $\rho=0$ is the zero map, we obtain
\[
\deg\big((u,v)-T_{\varepsilon}(1,u^+,v^+),B_R,0\big)=\deg\big((u,v),B_R,0\big)=1.
\]
\end{proof}

In order to prove that the solutions $(\eu ,\ev )$ of \eqref{2} we
are going to find are not bifurcating from the trivial solution
$(0,0)$, the next estimate will be crucial.
\begin{Lemma}\label{stima-gradiente}
Let $s>0$ be such that
\[s <
\min\left\{\displaystyle \frac{(p-1)(m-p)}{p}, \displaystyle
\frac{(q-1)(n-q)}{q}\right\}.\] Then, there exist two positive
constants $M_1$ and $M_2$ such that
\[
\left\|\nabla \eu^{\frac{(p-1)(m-1)-s}{p-1}}\right \|_{L^p(Q_T)} \le
M_1 \quad \text{and} \quad \left\|\nabla
\ev^{\frac{(q-1)(n-1)-s}{q-1}} \right\|_{L^q(Q_T)}\le M_2,
\]
for all solution pairs $(\eu  ,\ev   )$ of \eqref{eq:rhovarepsilon} and
 $\varepsilon >0$ sufficiently small.
\end{Lemma}
\begin{proof} Let $\displaystyle \gamma:= \frac{(p-1)(m-p)-ps}{p-1} >0.$
Multiplying the equation
\[
\begin{aligned}
\dfrac{\partial \eu  }{\partial t}&-\varepsilon\div(|\nabla
\eu|^{p-2}\nabla \eu) - \div(|\nabla \eu^m |^{p-2}\nabla \eu^m)
=\left (a(x,t)- \int_{\Omega}K_1(\xi,t)\eu  ^2(\xi, t-\tau_1)d\xi +\right.\\
&\left. + \int_{\Omega}K_2(\xi,t)\ev   ^2(\xi,
t-\tau_2)d\xi\right)\eu^{p-1}
\end{aligned}
\]
by $\eu^\gamma $, integrating over $Q_T$ and passing to the limit in
the Steklov averages, by the $T$-periodicity of $\eu$ we obtain
\[
\begin{aligned}
&\varepsilon \Int |\nabla \eu|^{p-2} \nabla \eu \nabla \eu^{\gamma} dxdt
+ \Int |\nabla \eu^m|^{p-2}\nabla \eu^m
\nabla \eu^\gamma dxdt \\
&= \Int a(x,t) \eu^{\gamma+p-1}(x,t) dxdt -
\int_0^T\left(\int_\Omega \eu^{\gamma+p-1}(x,t)dx\right)
\left(\int_{\Omega}K_1(\xi,t)\eu ^2(\xi,
t-\tau_1)d\xi\right)dt\\
&+ \int_0^T\left(\int_\Omega \eu^{\gamma+p-1}(x,t)dx\right) \left(
\int_{\Omega}K_2(\xi,t)\ev   ^2(\xi, t-\tau_2)d\xi\right)dt
\end{aligned}
\]
Now, since
\[
\Int |\nabla \eu|^{p-2}\nabla \eu \nabla \eu^{\gamma} dxdt = \gamma
\Int \eu^{\gamma-1}|\nabla \eu|^p dxdt \ge0,
\]
then
\begin{equation}\label{7}
\begin{aligned}
&\gamma m^{p-1} \left(\frac{p-1}{(p-1)(m-1)-s}\right)^p\left\|
\nabla \eu^{\frac{(p-1)(m-1)-s}{p-1}}\right\|_{L^p(Q_T)}^p = \gamma
m^{p-1}\Int \eu^{(p-1)(m-1)+\gamma -1}|\nabla \eu|^pdxdt\\
& \le \varepsilon \Int |\nabla \eu|^{p-2}\nabla \eu \nabla
\eu^{\gamma} dxdt + \Int |\nabla \eu^m|^{p-2}\nabla \eu^m \nabla
\eu^\gamma dxdt
\\& \le \|a\|_{L^\infty(Q_T)}\Int \eu^{\gamma+p-1}(x,t) dxdt - \int_0^T\left(\int_\Omega
\eu^{\gamma+p-1}(x,t)dx\right) \left(\int_{\Omega}K_1(\xi,t)\eu
^2(\xi,
t-\tau_1)d\xi\right)dt\\
&+ \int_0^T\left(\int_\Omega \eu^{\gamma+p-1}(x,t)dx\right) \left(
\int_{\Omega}K_2(\xi,t)\ev   ^2(\xi, t-\tau_2)d\xi\right)dt.
\end{aligned}
\end{equation}

Moreover, by the H\"{o}lder inequality with $\beta:=p\displaystyle
\frac{(p-1)(m-1)-s}{(p-1)(m-1)-ps}>1$ and the Poincar\'{e} inequality,
one has
\begin{equation}\label{8}
\begin{aligned}
\Int \eu^{\gamma+p-1} dxdt &\le |Q_T|^{\frac{1}{\beta'}}\left(\Int
\eu^{p\frac{(p-1)(m-1)-s}{p-1}}dxdt\right)^{\frac{1}{\beta}}=
|Q_T|^{\frac{1}{\beta'}}\left\|
\eu^{\frac{(p-1)(m-1)-s}{p-1}}\right\|_{L^p(Q_T)}^{\frac{p}{\beta}} \\
& \le |Q_T|^{\frac{1}{\beta'}}\left(
\frac{1}{\mu_p}\right)^{\frac{1}{\beta}}\left\| \nabla
\eu^{\frac{(p-1)(m-1)-s}{p-1}}\right\|_{L^p(Q_T)}^{\frac{p}{\beta}}.
\end{aligned}
\end{equation}
Here $\beta'$ is such that $\displaystyle \frac{1}{\beta}+
\frac{1}{\beta'} =1$. Thus \eqref{7} and \eqref{8} imply
\begin{equation}\label{7'}
\begin{aligned}
&\gamma m^{p-1}\left(\frac{p-1}{(p-1)(m-1)-s}\right)^p \left\|
\nabla \eu^{\frac{(p-1)(m-1)-s}{p-1}}\right\|_{L^p(Q_T)}^p \le
\|a\|_{L^\infty(Q_T)} |Q_T|^{\frac{1}{\beta'}}\left(
\frac{1}{\mu_p}\right)^{\frac{1}{\beta}}\left\| \nabla
\eu^{\frac{(p-1)(m-1)-s}{p-1}}\right\|_{L^p(Q_T)}^{\frac{p}{\beta}}
\\&- \int_0^T\left(\int_\Omega \eu^{\gamma+p-1}(x,t)dx\right)
\left(\int_{\Omega}K_1(\xi,t)\eu ^2(\xi,
t-\tau_1)d\xi\right)dt\\
&+ \int_0^T\left(\int_\Omega \eu^{\gamma+p-1}(x,t)dx\right) \left(
\int_{\Omega}K_2(\xi,t)\ev   ^2(\xi, t-\tau_2)d\xi\right)dt.
\end{aligned}
\end{equation}
By assumptions, there are constants $\ok_i \ge 0$, $i=2,3$, such
that $K_i(x,t) \le \ok_i$ for a.a. $(x,t) \in Q_T$. Thus, by
\eqref{8}, \eqref{7'} and Proposition \ref{R-moser}, we have
\[
\begin{aligned}
\gamma m^{p-1}&\left(\frac{p-1}{(p-1)(m-1)-s}\right)^p \left\|
\nabla \eu^{\frac{(p-1)(m-1)-s}{p-1}}\right\|_{L^p(Q_T)}^p \\
&\le
\|a\|_{L^\infty(Q_T)} |Q_T|^{\frac{1}{\beta'}}\left(\frac{1}{\mu_p}\right)^{\frac{1}{\beta}}
\left\| \nabla\eu^{\frac{(p-1)(m-1)-s}{p-1}}\right\|_{L^p(Q_T)}^{\frac{p}{\beta}}
+ \ok_2|\Omega|R^2 \Int \eu^{\gamma+p-1} dxdt \\
&\le (\|a\|_{L^\infty(Q_T)} + \ok_2|\Omega| R^2)
|Q_T|^{\frac{1}{\beta'}}\left(
\frac{1}{\mu_p}\right)^{\frac{1}{\beta}}\left\| \nabla
\eu^{\frac{(p-1)(m-1)-s}{p-1}}\right\|_{L^p(Q_T)}^{\frac{p}{\beta}}.
\end{aligned}
\]
In particular,
\[
\left\| \nabla \eu^{\frac{(p-1)(m-1)-s}{p-1}}\right\|_{L^p(Q_T)} \le
M_1,
\]
where
\[
M_1 := \left(\frac{(\|a\|_{L^\infty(Q_T)} + \ok_2|\Omega| R^2)
|Q_T|^{\frac{1}{\beta'}}\left(
\frac{1}{\mu_p}\right)^{\frac{1}{\beta}}[(p-1)(m-1)-s)]^p}{
[m(p-1)]^{p-1}[(p-1)(m-p)-ps)] }\right)^{\frac{\beta}{p(\beta-1)}}.
\]

Analogously, one can prove that
\[
\left\| \nabla \ev^{\frac{(q-1)(n-1)-s}{q-1}}\right\|_{L^q(Q_T)} \le
M_2,
\]
where
\[
M_2 := \left(\frac{(\|b\|_{L^\infty(Q_T)} + \ok_3|\Omega| R^2)
|Q_T|^{\frac{1}{\delta'}}\left(
\frac{1}{\mu_q}\right)^{\frac{1}{\delta}}[(q-1)(n-1)-s)]^q}{
[n(q-1)]^{q-1}[(q-1)(n-q)-qs)]
}\right)^{\frac{\delta}{q(\delta-1)}}.
\]
Here $\delta$ and $\delta'$ are such that $\delta:=q\displaystyle
\frac{(q-1)(n-1)-s}{(q-1)(n-1)-qs}$ and $\displaystyle
\frac{1}{\delta}+ \frac{1}{\delta'} =1$.
\end{proof}
\begin{rem}\label{rem:mn}
Observe that in the case when $ p = q = 2 $ a priori bounds for $ \| \nabla \eu^{m} \|_{ L^{2}( Q_{T} ) }, \| \nabla \ev^{n} \|_{ L^{2}( Q_{T} ) } $ have been
obtained in \cite{fnp} for sufficiently small $ \epsilon > 0 $ under the conditions that $ m, n > 1 $, i. e. in the case of slow diffusion.
Under the same condition $ m, n > 1 $ a priori bounds for $ \| \nabla \eu^{m} \|_{ L^{p}( Q_{T} ) }, \| \nabla \ev^{n} \|_{ L^{q}( Q_{T} ) } $ have been obtained
in \cite{fnp1} when $ p, q > 2 $, which again corresponds to the case of slow diffusion.
Therefore the assumptions $ m > p $ and $ n > q $ are required in Lemma~\ref{stima-gradiente} only for the singular case $ p, q \in ( 1 , 2 ) $, which, as already noticed,
allows the fast diffusion if $ p, q \in \left( 1, \frac{ 1 + \sqrt{5} }{ 2 } \right) $.
Finally, observe that, if $ p, q > 1 $, we have the fast diffusion when $ m < 1/(p-1) $ and $ n < 1/(q-1) $: to the best of our knowledge, this case is not treated in the existing literature devoted to this problem.
\end{rem}
The following result guarantees that the foreseen solutions $(\eu
,\ev )$ of \eqref{2} are not bifurcating from the trivial solution
$(0,0)$ as $\varepsilon$ ranges in $(0,\varepsilon_0)$, where
$\varepsilon_0$ is such that
\begin{equation}\label{varepsilon0}
\begin{aligned}
\theta(C_1, C_2):= \min&\left\{\frac{1}{T}\Int  a(x,t)e_p^p(x)
dxdt -\varepsilon_0 \mu_p-\dfrac{\uk_2C_2}{T}, \right.\\
&\left. \frac{1}{T}\Int b(x,t)e_q^q(x) dxdt -\varepsilon_0
\mu_q-\dfrac{\uk_3C_1}{T}\right\} >0,
\end{aligned}
\end{equation}
where $\mu_p$, $\mu_q$, $e_p$, $e_q$, $\uk_2$ and $\uk_3$ are as in
Hypotheses \ref{ipotesi1}.

To this aim let
\[
\begin{aligned}
r_0:=\min &\left\{\left(\dfrac{\Int  a(x,t) e_p^p(x)dxdt -
\varepsilon_0T\mu_p}{D_1}\right)^{\frac{1}{2}}\!, \left(\dfrac{\Int
 a(x,t)e_p^p(x) dxdt -
\varepsilon_0T\mu_p}{D_1}\right)^{\frac{1}{s}}\!,\right.\\
&\left. \quad\left(\dfrac{\Int b(x,t)e_q^q(x) dxdt -
\varepsilon_0T\mu_q}{D_2}\right)^{\frac{1}{2}}\!, \left(\dfrac{\Int
 b(x,t)e_q^q(x) dxdt -
\varepsilon_0T\mu_q}{D_2}\right)^{\frac{1}{s}}\right\},
\end{aligned}
\]
where
\[
D_1:=\|K_1\|_{L^1(Q_T)}+ \|K_2\|_{L^1(Q_T)} + \displaystyle p
\|e_p^{p-1}\|_{L^\infty(\Omega)}\|\nabla e_p\|_{L^\infty(\Omega)}
 \left(\frac{m(p-1)M_1}{(p-1)(m-1)-s} \right)^{p-1}
|Q_T|^{\frac{1}{p}},
\]
\[
D_2:=\|K_3\|_{L^1(Q_T)}+ \|K_4\|_{L^1(Q_T)} + \displaystyle q
\|e_q^{q-1}\|_{L^\infty(\Omega)}\|\nabla e_q\|_{L^\infty(\Omega)}
 \left(\frac{n(q-1)M_2}{(q-1)(n-1)-s} \right)^{q-1}
|Q_T|^{\frac{1}{q}}
\]
and $M_1$, $M_2$, $s$ are as in Lemma \ref{stima-gradiente}. By
\eqref{varepsilon0}, $r_0$ is obviously positive.
\begin{Proposition}\label{r}
For all solution pairs $(\eu  ,\ev   )$ of \eqref{eq:rhovarepsilon} and
all $\varepsilon \in (0, \varepsilon_0)$, it results
\[
\max\{\|\eu   \|_{L^\infty(Q_T)}, \|\ev   \|_{L^\infty(Q_T)}\} \ge
r_0.
\]
Moreover $\deg\big(\,(u, v)- T_{\varepsilon}(1, u^+, v^+),
B_r,0\big)=0$ for all $r\in(0,r_0)$.
\end{Proposition}
\begin{proof}
By contradiction, assume that for some $r\in(0,r_0)$ there exists a
pair $(\eu  , \ev   )\neq (0,0)$ such that $(\eu  ,\ev
)=G_{\varepsilon }\big(1, f(\eu ^+,\ev ^+), g(\eu ^+,\ev ^+)\big)$
with $\|\eu \|_{L^\infty(Q_T)}\le r$ and $\|\ev \|_{L^\infty(Q_T)}
\le r$. Assume that $\eu  \neq 0$ and take $\phi \in
C_c^{\infty}(\Omega)$. Since by Proposition \ref{positivity} we have
$\eu  >0$ in $Q_T$, we can multiply the equation
\[
\begin{aligned}
&\dfrac{\partial \eu  }{\partial t}-\varepsilon\div(|\nabla \eu
|^{p-2}\nabla \eu) - \div(|\nabla \eu^m |^{p-2}\nabla \eu^m)=\\=&
\left (a(x,t)- \int_{\Omega}K_1(\xi,t)\eu  ^2(\xi, t-\tau_1)d\xi \right.
\left. + \int_{\Omega}K_2(\xi,t)\ev   ^2(\xi,
t-\tau_2)d\xi\right)\eu^{p-1}
\end{aligned}
\]
by $\dfrac{\phi^p}{\eu^{p-1}  }$, integrate over $Q_T$ and pass to
the limit in the Steklov averages in order to obtain
\begin{equation}\label{6}
\begin{aligned}
&- \varepsilon\Int  \frac{\phi^p}{\eu^{p-1}  } \div(|\nabla \eu
|^{p-2}\nabla \eu)dxdt - \Int \frac{\phi^p}{\eu^{p-1} }\div(|\nabla
\eu^m |^{p-2}\nabla \eu^m)dxdt= \Int \!\!\phi^p(x) a(x,t) dxdt-\\& -
\Int \!\!\!\!\phi^p(x)\left(\int_{\Omega}K_1(\xi,t)\eu ^2(\xi,
t-\tau_1)d\xi \right)dxdt +\Int
\!\!\!\!\phi^p(x)\left(\int_{\Omega}K_2(\xi,t)\ev   ^2(\xi,
t-\tau_2)d\xi \right)dxdt,
\end{aligned}
\end{equation}
by the $T$-periodicity of $\eu  $. By the generalized Picone's identity
due to Allegretto-Huan, see \cite{ah}, one has
\begin{equation}\label{6'}
\begin{aligned}
-\varepsilon \!\!\Int \!\!\frac{\phi^p}{\eu^{p-1}  }  \div(|\nabla \eu
|^{p-2}\nabla \eu) dxdt&= \varepsilon \!\!\Int|\nabla \eu |^{p-2}
\nabla\eu \nabla \left(\frac{\phi^p}{\eu^{p-1}  }\right)  dxdt\le
\varepsilon\!\!\Int|\nabla \phi |^pdxdt.
\end{aligned}
\end{equation}
Indeed, we have that
\[
\begin{aligned}
|\nabla \eu|^{p - 2} \nabla \eu \nabla \left( \dfrac{\phi^p}{\eu^{p -1}} \right)
&\le p |\nabla \phi| \left( \dfrac{\phi}{\eu} |\nabla \eu| \right)^{p - 1} - (p - 1)\left( \dfrac{\phi}{\eu} |\nabla \eu| \right)^p\\
&= \left( \dfrac{\phi}{\eu} |\nabla \eu| \right)^p + p \left( \dfrac{\phi}{\eu} |\nabla \eu| \right)^{p - 1} \left( |\nabla \phi| - \dfrac{\phi}{\eu} |\nabla \eu| \right)\\
&\le |\nabla \phi|^p,
\end{aligned}
\]
since the function $\R \ni \xi \mapsto |\xi|^p$ is convex for $p>1$.

Moreover,
\begin{equation}\label{6''}
\begin{aligned}
&- \Int \frac{\phi^p}{\eu^{p-1} }\div(\nabla \eu^m|\nabla \eu^m
|^{p-2})dxdt= \Int  \nabla \left(\frac{\phi^p}{\eu^{p-1} }\right) \nabla
\eu^m|\nabla
\eu^m |^{p-2} dxdt \\
&= m^{p-1}\Int \eu^{(m-1)(p-1)}|\nabla \eu|^{p-2} \nabla \eu
\frac{\eu^{p-1}\nabla \phi^p -\phi^p \nabla \eu^{p-1}}{\eu^{2(p-1)}}dxdt \\
&= pm^{p-1}\Int \phi^{p-1} \eu^{(p-1)(m-2)}|\nabla \eu|^{p-2} \nabla
\eu \nabla \phi dxdt - m^{p-1}(p-1)\Int \phi^p
\eu^{(m-2)(p-1)-1}|\nabla \eu|^{p}
dxdt\\
&\le p \|\phi^{p-1}\|_{L^\infty(\Omega)}\|\nabla
\phi\|_{L^\infty(\Omega)} m^{p-1} \Int |\nabla \eu|^{p-1}
\eu^{(p-1)(m-2)}dxdt.
\end{aligned}
\end{equation}
By \eqref{6}-\eqref{6''}, it follows
\[
\begin{aligned}
 &\Int
\!\!\phi^p(x) a(x,t) dxdt - \Int
\!\!\!\!\phi^p(x)\left(\int_{\Omega}K_1(\xi,t)\eu ^2(\xi,
t-\tau_1)d\xi \right)dxdt \\
&+\Int \!\!\!\!\phi^p(x)\left(\int_{\Omega}K_2(\xi,t)\ev   ^2(\xi,
t-\tau_2)d\xi \right)dxdt\\
& \le p \|\phi^{p-1}\|_{L^\infty(\Omega)}\|\nabla
\phi\|_{L^\infty(\Omega)} m^{p-1} \Int |\nabla \eu|^{p-1}
\eu^{(p-1)(m-2)}dxdt +\varepsilon\!\!\Int |\nabla \phi |^pdxdt.
\end{aligned}\]
Taking $\phi(x) = \phi_j(x) \rightarrow e_p(x)$ in $C^1_0(\Omega)$
as $j \rightarrow + \infty$ and since $\varepsilon < \varepsilon_0$, one
has
\[
\begin{aligned}
 &\Int
\!\!e_p^p(x) a(x,t) dxdt - \Int
\!\!\!\!e_p^p(x)\left(\int_{\Omega}K_1(\xi,t)\eu ^2(\xi,
t-\tau_1)d\xi \right)dxdt \\
&+\Int \!\!\!\!e_p^p(x)\left(\int_{\Omega}K_2(\xi,t)\ev   ^2(\xi,
t-\tau_2)d\xi \right)dxdt\\
& \le p \|e_p^{p-1}\|_{L^\infty(\Omega)}\|\nabla
e_p\|_{L^\infty(\Omega)} m^{p-1} \Int |\nabla \eu|^{p-1}
\eu^{(p-1)(m-2)}dxdt +\varepsilon_0\!\!\Int |\nabla e_p|^pdxdt.
\end{aligned}\]
Taking into account that $\|e_p\|_{L^p(\Omega)}=1$, the previous
inequality imply
\begin{equation}\label{10}
\begin{aligned}
&\Int e_p^p(x) a(x,t) dxdt -\varepsilon_0T\mu_p \le  \Int
K_1(\xi,t)\eu ^2(\xi, t-\tau_1)d\xi dt - \Int K_2(\xi,t)\ev
^2(\xi, t-\tau_2)d\xi dt \\&+ p
\|e_p^{p-1}\|_{L^\infty(\Omega)}\|\nabla e_p\|_{L^\infty(\Omega)}
m^{p-1} \Int |\nabla \eu|^{p-1} \eu^{(p-1)(m-2)}dxdt.
\end{aligned}
\end{equation}

Now we estimate the term $\Int |\nabla \eu|^{p-1}
\eu^{(p-1)(m-2)}dxdt$. Since $\|\eu \|_{L^\infty(Q_T)}\le r$, using
the H\"{o}lder inequality, one has
\begin{equation}\label{11}
\begin{aligned}
&\Int |\nabla \eu|^{p-1} \eu^{(p-1)(m-2)}dxdt \le r^s  \Int |\nabla
\eu|^{p-1} \eu^{(p-1)(m-2)-s}dxdt \\
&= r^s \left(\frac{p-1}{(p-1)(m-1)-s} \right)^{p-1}\Int \left|\nabla
\eu ^{\frac{(p-1)(m-1)-s}{p-1}}\right|^{p-1} dxdt \\
&\le r^s \left(\frac{p-1}{(p-1)(m-1)-s} \right)^{p-1}
|Q_T|^{\frac{1}{p}} \left \|\nabla \eu
^{\frac{(p-1)(m-1)-s}{p-1}}\right\|_{L^p(Q_T)}^{p-1}.
\end{aligned}
\end{equation}
 Observe that $(p-1)(m-1)-s>0$, since, by
assumption, $s <
\displaystyle \frac{(p-1)(m-p)}{p} <(p-1)(m-1)$. \\
By Lemma \ref{stima-gradiente}, \eqref{10} and \eqref{11}, it
follows
\[
\begin{aligned}
&\Int e_p^p(x) a(x,t) dxdt -\varepsilon_0T\mu_p \le
(\|K_1\|_{L^1(Q_T)}+ \|K_2\|_{L^1(Q_T)})r^2\\
& + p \|e_p^{p-1}\|_{L^\infty(\Omega)}\|\nabla
e_p\|_{L^\infty(\Omega)}
 \left(\frac{m(p-1)M_1}{(p-1)(m-1)-s} \right)^{p-1}
|Q_T|^{\frac{1}{p}}r^s\\
& \le (\|K_1\|_{L^1(Q_T)}+ \|K_2\|_{L^1(Q_T)} + C)\max\{r^2, r^s\},
\end{aligned}
\]
where $C:= \displaystyle p \|e_p^{p-1}\|_{L^\infty(\Omega)}\|\nabla
e_p\|_{L^\infty(\Omega)}
 \left(\frac{m(p-1)M_1}{(p-1)(m-1)-s} \right)^{p-1}
|Q_T|^{\frac{1}{p}}$.
\\
 Thus, if $\max\{r^2, r^s\} =r^2$, then
\[\begin{aligned}
r_0\le\left(\dfrac{\Int e_p^p (x)
a(x,t)dxdt-\varepsilon_0\mu_pT}{\|K_1\|_{L^1(Q_T)}+ \|K_2\|_{L^1(Q_T)}
+ C}\right)^{\frac{1}{2}}\le r,
\end{aligned}\]
that is a contradiction; analogously if $\max\{r^2, r^s\} =r^s$. The
same argument applies if $\ev \neq 0$.  Fix any $r\in(0,r_0)$. The
result above shows that
\[
(u,v)\not= G_\varepsilon (\sigma, f(u^+,v^+)+(1-\sigma), g(u^+,v^+) +(1-\sigma)),
\]
for all $(u,v)\in \partial B_r$ and all $\sigma\in [0,1]$. From the homotopy invariance
of the Leray-Schauder degree, we have
\[
\deg((u,v)-T_\varepsilon(1,u^+,v^+), B_r,
0)=\deg((u,v)-G_\varepsilon(0,f(u^+,v^+)+1,g(u^+,v^+)+1),B_r,0).
\]
The last degree is zero since the equation
\[
(u,v)=G_\varepsilon(0,f(u^+,v^+)+1,g(u^+,v^+)+1)
\]
admits neither trivial nor trivial solution in $B_r$, since $r <
r_0$.
\end{proof}
The next result is our general coexistence result for \eqref{1}.
\begin{Theorem}\label{thm:generale}
Problem \eqref{1} has a $T$-periodic non-negative solution $(u,v)$
with both non-trivial $u,v$.
\end{Theorem}
\begin{proof}
By Propositions~\ref{R-moser} and \ref{r} and the excision property
of the topological degree, there are $R>r>0$, independent of
$\varepsilon$, such that
\[
\deg\big((\eu,\ev)-G_{\varepsilon}(1,f(\eu^+,\ev^+),g(\eu^+,\ev^+)),B_R\setminus\overline{B}_r,0\big)=1,
\]
for any $\varepsilon \in(0,\varepsilon_0)$.

Let us fix any $\varepsilon  \in(0,\varepsilon_0)$. There is
$\sigma_0=\sigma_0(\varepsilon)\in(0,1)$ such that still
\[
\deg\big((\eu,\ev)-G_{\varepsilon}(\sigma,f(\eu^+,\ev^+)+(1-\sigma),g(\eu^+,\ev^+)+(1-\sigma)),B_R\setminus\overline{B}_r,0\big)
=1
\]
for all $\sigma\in[\sigma_0,1]$, by the continuity of Leray-Schauder
degree. This implies that the set of solution triples
$(\sigma,\eu,\ev)\in[0,1]\times(B_R\setminus\overline{B}_r)$ such
that
\begin{equation}\label{eq:sigma}
(\eu,\ev)=G_{\varepsilon\eta}\big(\sigma,f(\eu^+,\ev^+)+(1-\sigma),g(\eu^+,\ev^+)+(1-\sigma)\big)
\end{equation}
contains a continuum $\mathcal S_{\varepsilon}$ with the property that
\[
\mathcal
S_{\varepsilon}\cap\left[\{\sigma\}\times\left(B_R\setminus\overline{B}_r\right)\right]\ne\emptyset
\qquad\text{for all }\sigma\in[\sigma_0,1].
\]
Now, all the pairs $(\eu,\ev)$ such that $(1,\eu,\ev)\in\mathcal
S_{\varepsilon}$ are $T$-periodic solutions of \eqref{2} with
$(\eu,\ev)\neq (0,0)$ and, hence, satisfy
\eqref{eq:ipotesimaggiorazioni}. Since the $L^2$-norm is continuous
with respect to the $L^\infty$-norm and $\mathcal S_{\varepsilon}$ is a
continuum, for every $\nu>0$ there is
$\sigma_\nu\in\left[\sigma_0,1\right)$ such that
\[
\|\eu\|_{L^2(Q_T)}^2\le C_1+\nu\qquad\text{and}\qquad
\|\ev\|_{L^2(Q_T)}^2\le C_2+\nu
\]
for all $(\eu,\ev)$ with $(\sigma,\eu,\ev)\in\mathcal S_{\varepsilon}$
and $\sigma\in[\sigma_\nu,1]$. Observe that, if
$(\sigma,\eu,\ev)\in\mathcal S_{\varepsilon}$ for $\sigma<1$, then
$\eu$ and $\ev$ are {\em positive} solutions of \eqref{eq:sigma}.
 Moreover, if $\nu$ is sufficiently small, then we still have
$\theta(C_1+\nu,C_2+\nu)>0$.

Now, setting
\[
\begin{aligned}
K_p:=&\left[\|K_1\|_{L^1(Q_T)}+ \displaystyle p
\|e_p^{p-1}\|_{L^\infty(\Omega)}\|\nabla e_p\|_{L^\infty(\Omega)}
 \left(\frac{m(p-1)M_1}{(p-1)(m-1)-s} \right)^{p-1}
|Q_T|^{\frac{1}{p}}\right] \\
K_q:=&\left[\|K_4\|_{L^1(Q_T)}+ \displaystyle q
\|e_q^{q-1}\|_{L^\infty(\Omega)}\|\nabla e_q\|_{L^\infty(\Omega)}
 \left(\frac{n(q-1)M_2}{(q-1)(n-1)-s} \right)^{q-1}
|Q_T|^{\frac{1}{q}}\right],
\end{aligned}
\]
we can prove that, if $\nu$ is sufficiently small, then
\[
\begin{aligned}
\|\eu\|_{L^\infty(Q_T)},\|\ev\|_{L^\infty(Q_T)}\ge \min&\left\{
\left(\frac{T\theta(C_1+\nu,C_2+\nu)}{K_p}\right)^{\frac{1}{2}},
\left(\frac{T\theta(C_1+\nu,C_2+\nu)}{K_p}\right)^{\frac{1}{s}}
\right.\\
&\left.\left(\frac{T\theta(C_1+\nu,C_2+\nu)}{K_q}\right)^{\frac{1}{2}},
\left(\frac{T\theta(C_1+\nu,C_2+\nu)}{K_q}\right)^{\frac{1}{s}}
\right\}=:\lambda_\nu
\end{aligned}
\]
for all $\eu,\ev$ such that $(\sigma,\eu,\ev)\in\mathcal
S_{\varepsilon}$ and $\sigma\in\left[\sigma_\nu,1\right)$.
 Indeed, let $(\eu,\ev)$
be a solution of \eqref{eq:sigma}. Arguing by contradiction, assume
that $\|\eu\|_{L^\infty(Q_T)}<\lambda_\nu$ and proceeding as in the
proof of Proposition~\ref{r} (recall that $\eu>0$ since $(\eu,\ev)$
solves \eqref{eq:sigma} with $\sigma<1$) we obtain the inequality
\[
\begin{aligned}
\iint_{Q_T}e_p^p(x)a(x,t)dxdt-\varepsilon_0\mu_pT< \max\{\lambda_\nu^2,
\lambda_\nu^s\} K_p +\uk_2(C_2+\nu).
\end{aligned}
\]
 Thus, if $\max\{\lambda_\nu^2, \lambda_\nu^s\} =\lambda_\nu^2$, using the definition of $\theta$,
one has
\[
T\theta(C_1+\nu,C_2+\nu)\le\iint_{Q_T}e_p^p(x)a(x,t)dxdt-\varepsilon_0\mu_pT-\uk_2(C_2+\nu)<
\lambda_\nu^2K_p,
\]
that is
\[
\left(\frac{T\theta(C_1+\nu,C_2+\nu)}{K_p}\right)^{\frac{1}{2}}
<\lambda_\nu
\]
 which is a contradiction with the definition of
$\lambda_\nu$; analogously if $\max\{\lambda_\nu^2, \lambda_\nu^s\}
=\lambda_\nu^s$. The same argument shows that
$\|\ev\|_{L^\infty(Q_T)}\ge\lambda_\nu$.

Now, if we let $\sigma\to 1$ and $\nu\to 0$, we obtain that
\eqref{2} has at least a solution $(\eu,\ev)$ such that
$\|\eu\|_{L^\infty(Q_T)},\|\ev\|_{L^\infty(Q_T)}\ge\lambda_0$, since
$\mathcal S_{\varepsilon}$ is a continuum and
$\lambda_\nu\to\lambda_0$ as $\nu\to 0$.

Finally, we show that a solution $(u,v)$ of \eqref{1} with both
non-trivial $u,v\ge 0$ is obtained as a limit of $(\eu,\ev)$ as
$\varepsilon\to 0$, since $\lambda_0$ is independent of
$\varepsilon$.

Since $\eu,\ev$ are H\"{o}lder continuous in $\overline{Q}_T$,
bounded in $C(\overline{Q}_T)$ uniformly in $\varepsilon>0$ and the
structure conditions of \cite{i2} and \cite{i3} are satisfied for
the equations of system \eqref{2}, whenever $\varepsilon \in\left(0,
\varepsilon_0\right)$, \cite[Theorem 1.1]{i2} and \cite[Theorem
1.3]{i3} apply to conclude that the inequality
\[
|\eu (x_1, t_1) - \eu  (x_2, t_2))| \le \Gamma
(|x_1-x_2|^\beta+|t_1- t_2|^{\frac{\beta}{p}})
\]
holds for any $(x_1,t_1),(x_2,t_2)\in\overline{Q}_T$, where the
constants $\Gamma >0$ and $\beta\in(0,1)$ are independent of
$\|\eu\|_{L^\infty(Q_T)}$. The same inequality holds for $\ev$.
Therefore, by the Ascoli-Arzel\`{a} Theorem, a subsequence of
$(\eu,\ev)$ converges uniformly in $\overline{Q}_T$ to a pair
$(u,v)$ satisfying
\[
\lambda_0\le \|u\|_{L^\infty(Q_T)},\|v\|_{L^\infty(Q_T)} \le R.
\]
 Moreover, from \eqref{5} we have that $\eu$ satisfies the inequality
\begin{equation}\label{13}
\dfrac{\partial \eu  }{\partial t}-\varepsilon\div(|\nabla \eu
|^{p-2}\nabla \eu)- \div(|\nabla \eu^m |^{p-2}\nabla \eu^m) \le
K\eu^{p-1},\quad\text{in } \; Q_T,
\end{equation}
where $K$ is a positive constant independent of $\varepsilon$.
Multiplying \eqref{13} by $\eu^m$, integrating over $Q_T$ and
passing to the limit in the Steklov averages $(\eu)_h$, one has
\begin{equation}\label{costante}
\begin{aligned}
 \iint_{Q_T}
|\nabla \eu^m|^{p}dxdt& \le
 \varepsilon m\Int u^{m-1}|\nabla \eu|^p dxdt + \Int |\nabla \eu ^m|^p dxdt  \\
&= \varepsilon \Int |\nabla \eu|^{p-2}\nabla \eu \nabla \eu^m dxdt + \Int
\nabla \eu ^m \nabla \eu^m |\nabla \eu^m|^{p-2}dxdt
\\
&\le K \iint_{Q_T}\eu^{p+m-1}dxdt \le M,
\end{aligned}\end{equation}
by the $T$-periodicity of $\eu$, its non-negativity and its
boundedness in $L^{\infty}(Q_T)$. Here $M$ is positive and
independent of $\varepsilon$. An analogous estimate holds for $\ev$.
Thus the sequences $\eu^m, \ev^n$  are uniformly bounded in
$L^{p}\big(0,T;W^{1,p}_0(\Omega)\big)$ and in
$L^{q}\big(0,T;W^{1,q}_0(\Omega)\big)$, respectively. Thus, up to
subsequence if necessary, $(\eu^m,\ev^n)$ converges weakly in
$L^{p}\big(0,T;W^{1,p}_0(\Omega)\big)\times
L^{q}\big(0,T;W^{1,q}_0(\Omega)\big)$ and strongly in
$C(\overline{Q}_T)\times C(\overline{Q}_T)$ to $(u^m,v^n)$. In
particular $(u^m,v^n)\in L^{p}\big(0,T;W^{1,p}_0(\Omega)\big) \times
L^{q}\big(0,T;W^{1, q}_0(\Omega)\big)$. We finally claim that the
pair $(u,v)$ satisfies the identities
\[
\begin{aligned}
0=&\iint_{Q_T}\left\{-u\frac{\partial\varphi}{\partial t}
+|\nabla u^m|^{p-2}\nabla u^m\cdot\nabla\varphi-au^{p-1}\varphi\right.\\
&\left.+u^{p-1}\varphi\int_{\Omega}[K_1(\xi,t)u^2(\xi,t-\tau_1)
-K_2(\xi,t)v^2(\xi,t-\tau_2)]d\xi\right\}dxdt
\end{aligned}
\]
and
\[
\begin{aligned}
0=&\iint_{Q_T}\left\{-v\frac{\partial\varphi}{\partial t}
+|\nabla v^n|^{q-2}\nabla v^n\cdot\nabla\varphi-bv^{q-1}\varphi\right.\\
&\left.+v^{q-1}\varphi\int_{\Omega}[-K_3(\xi,t)u^2(\xi,t-\tau_3)
+K_4(\xi,t)v^2(\xi,t-\tau_4)]d\xi\right\}dxdt,
\end{aligned}
\]
for any $\varphi\in C^1(\overline{Q}_T)$ such that
$\varphi(x,T)=\varphi(x,0)$ for any $x\in\Omega$ and
$\varphi(x,t)=0$ for any $(x,t)\in
\partial\Omega\times[0,T]$, that is $(u,v)$ is a generalized
solution of \eqref{1}. The approach for doing this is standard, in
the sequel we write it in detail for the reader's convenience. First
of all, observe that
\begin{equation}\label{limite1}
\lim_{\varepsilon \rightarrow 0} \varepsilon\Int |\nabla \eu
|^{p-2}\nabla \eu\nabla \varphi dxdt =0
\end{equation}
for all test functions $\varphi$. In fact, multiplying the equation
\[
\begin{aligned}
\dfrac{\partial \eu  }{\partial t}&-\varepsilon\div(|\nabla \eu
|^{p-2}\nabla \eu)- \div(|\nabla \eu^m |^{p-2}\nabla \eu^m)
=\left (a(x,t)- \int_{\Omega}K_1(\xi,t)\eu  ^2(\xi, t-\tau_1)d\xi +\right.\\
&\left. + \int_{\Omega}K_2(\xi,t)\ev   ^2(\xi,
t-\tau_2)d\xi\right)\eu^{p-1}
\end{aligned}
\]
by $\eu$, integrating over $Q_T$, using the $T$-periodicity of $\eu$
and its non-negativity and passing, as $h\to0,$ to the limit in the Steklov
averages $(\eu)_h$, we obtain
\[
\begin{aligned}
\|\sqrt[p]{\varepsilon}\nabla \eu\|_{L^p(Q_T)}^p &= \varepsilon \Int
|\nabla \eu|^p dxdt \le \varepsilon \Int |\nabla \eu|^p dxdt+ m^{p-1}\Int
\eu^{(m-1)(p-1)}|\nabla \eu|^p dxdt \\
&\le \varepsilon \Int |\nabla \eu|^p dxdt + \Int|\nabla \eu^m|^{p-2}  \nabla
\eu^m \nabla \eu dxdt \le C,
\end{aligned}
\]
where $C:= (\|a\|_{L^\infty (Q_T)} + \|K_2\|_{L^\infty(Q_T)}R^2)
|Q_T|^{1-\frac{p}{2}}C_1^{\frac{p}{2}}$ (recall that, by assumption,
being $p<2$, $\|\eu\|_{L^p(Q_T)}\le
|Q_T|^{\frac{1}{p}-\frac{1}{2}}\|\eu\|_{L^2(Q_T)}\le |Q_T|^{\frac{1}{p}-\frac{1}{2}}\sqrt{C_1}$).
Thus, by the H\"{o}lder inequality,
\[
\begin{aligned}
 \left|\varepsilon\Int | \nabla \eu |^{p-2}\nabla \eu \nabla \varphi
dxdt\right|  & \le \Int \varepsilon^{\frac{1}{p'}}|\nabla \eu|^{p-1}
\varepsilon^{\frac{1}{p}}|\nabla \varphi| dxdt\\& \le
\|\sqrt[p]{\varepsilon} \nabla \eu\|_{L^p(Q_T)}^{\frac{p}{p'}}
\varepsilon^{\frac{1}{p}}\|\nabla \varphi\|_{L^p(Q_T)} \le
\sqrt[p]{\varepsilon} \sqrt[p']{C }\|\nabla
\varphi\|_{L^p(Q_T)}\rightarrow 0
\end{aligned}
\]
as $\varepsilon \rightarrow 0$, for all test functions $\varphi$.

In what follows we will prove that
\begin{equation}\label{limite2}
\lim_{\varepsilon \rightarrow 0}\displaystyle \Int |\nabla
\eu^m|^{p-2}\nabla \eu^m\cdot\nabla\varphi dxdt= \Int |\nabla
u^m|^{p-2}\nabla u^m\cdot\nabla\varphi dxdt,
\end{equation}
for all test functions $\varphi$. To this aim, observe that
$\displaystyle |\nabla \eu ^m|^{p-2}\nabla \eu^m $ is bounded in
$\left(L^{\frac{p}{p-1}}(Q_T)\right)^N$. In fact,
\[
\Int \left||\nabla \eu ^m|^{p-2}\nabla \eu^m \right|^{\frac{p}{p-1}} dxdt
= \Int |\nabla \eu^m |^p dxdt \le M,
\]
as proved in \eqref{costante}. Thus there exists
 $H \in \big(L^{\frac{p}{p-1}}(Q_T)\big)^N$ such
that $|\nabla \eu ^m|^{p-2}\nabla \eu^m$ weakly converges to $H$ in
$\big(L^{\frac{p}{p-1}}(Q_T)\big)^N$ as $\varepsilon \rightarrow 0$.
Now, using \eqref{limite1}, it is easy to prove that
\begin{equation}\label{densita}
\begin{aligned}
0= &\iint_{Q_T}\left\{-u\frac{\partial\varphi}{\partial t}
 +H\cdot\nabla\varphi-au^{p-1}\varphi\right.\\
&\left.+u^{p-1}\varphi\int_{\Omega}[K_1(\xi,t)u^2(\xi,t-\tau_1)
-K_2(\xi,t)v^2(\xi,t-\tau_2)]d\xi\right\}dxdt
\end{aligned}
\end{equation}
for any $\varphi\in C^1(\overline{Q}_T)$ such that
$\varphi(x,T)=\varphi(x,0)$ for any $x\in\Omega$ and
$\varphi(x,t)=0$ for any $(x,t)\in\partial\Omega\times[0,T]$ (and,
by density, for any $T$-periodic $\varphi\in
L^p\big(0,T;W^{1,p}_0(\Omega)\big)\cap C(\overline{Q_T}))$. For this
it remains to prove that for every $\varphi\in C^1(\overline{Q}_T)$
\begin{equation}\label{uguaglianza}
\iint_{Q_T}|\nabla u^m|^{p-2}\nabla u^m \cdot\nabla \varphi dxdt=
\iint_{Q_T}H \cdot\nabla \varphi dxdt.
\end{equation}
Consider the vector function $H(Y):=|Y|^{p-2}Y$. Then
\[
H'(Y)=|Y|^{p-2}I+(p-2)|Y|^{p-4}Y\otimes Y
\]
is a positive definite matrix and, taken $w\in
L^p\big(0,T;W^{1,p}_0(\Omega)\big)$, there exists a vector $Y$ such
that
\[
\begin{aligned}
0\le &\langle H'(Y)(\nabla \eu^m -\nabla w),\nabla \eu^m -\nabla w\rangle\\
=&\langle H(\nabla \eu^m )-H(\nabla w),\nabla \eu^m -\nabla
w\rangle.
\end{aligned}
\]
The previous inequality implies
\[
\iint_{Q_T}\left\{ |\nabla \eu ^m|^{p-2}\nabla \eu^m-|\nabla
w|^{p-2}\nabla w \right\} \cdot\nabla(\eu^m -w)dxdt\ge0,
\]
for all $w\in L^p\big(0,T;W^{1, p}_0(\Omega)\big)$, that is
\[
\iint_{Q_T} |\nabla \eu ^m|^p dxdt- \iint_{Q_T} |\nabla \eu ^m|^{p-2}
\nabla \eu^m\cdot\nabla w dxdt- \Int |\nabla w|^{p-2}\nabla w
\cdot\nabla(\eu^m -w)dxdt\ge0,
\]
for all $w\in L^p\big(0,T;W^{1, p}_0(\Omega)\big)$. As in
\eqref{costante}, one has
\[
\begin{aligned}
 \iint_{Q_T}
|\nabla \eu^m|^{p}dxdt& \le \varepsilon m\Int u^{m-1}|\nabla \eu|^p dxdt+ \Int
|\nabla \eu ^m|^p dxdt
\\
&\le
\iint_{Q_T}\left[a-\int_{\Omega}K_1(\xi,t)\eu^2(\xi,t-\tau_1)d\xi+\int_{\Omega}K_2(\xi,t)\ev^2(\xi,t-\tau_2)d\xi\right]
\eu^{p+m-1}dxdt.
\end{aligned}
\]
Thus, from the previous two inequalities, we obtain
\[
\begin{aligned}
& \iint_{Q_T} |\nabla \eu ^m|^{p-2}\nabla \eu^m \cdot\nabla w dxdt+
\Int|\nabla w|^{p-2}\nabla w  \cdot\nabla(\eu^m -w) dxdt \\
&\le
\iint_{Q_T}\left[a-\int_{\Omega}K_1(\xi,t)\eu^2(\xi,t-\tau_1)d\xi+\int_{\Omega}K_2(\xi,t)\ev^2(\xi,t-\tau_2)d\xi\right]
\eu^{p+m-1}dxdt.
\end{aligned}
\]
Letting $\varepsilon\rightarrow 0$ and using \eqref{costante}, we
have
\begin{equation}\label{weak}
\begin{aligned}
&\iint_{Q_T}\left[H\cdot\nabla w
 +|\nabla w|^{p-2}\nabla w\cdot\nabla(u^m-w)\right]dxdt\\
\le &\iint_{Q_T}\left[a-\int_{\Omega}K_1(\xi,t)u^2(\xi,t-\tau_1)d\xi
 +\int_{\Omega}K_2(\xi,t)v^2(\xi,t-\tau_2)d\xi\right]u^{p+m-1}dxdt.
\end{aligned}
\end{equation}
(Observe that, being $p>1$, $\nabla \eu^m$ is also bounded in
$L^1(Q_T$.)

On the other hand, by density we can take $u^m= \varphi$ in
\eqref{densita} and obtain
\[
\begin{aligned}
&\iint_{Q_T}H\cdot\nabla u^m dxdt\\
=&\iint_{Q_T}\left[a-\int_{\Omega}K_1(\xi,t)u^2(\xi, t-\tau_1)d\xi
+\int_{\Omega}K_2(\xi,t)v^2(\xi, t-\tau_2)d\xi\right]u^{p+m-1}dxdt.
\end{aligned}
\]
This equality together with \eqref{weak} implies
\begin{equation}\label{h}
0\le \iint_{Q_T}(H -|\nabla w|^{p-2}\nabla w)\cdot \nabla(u^m-w) dxdt.
\end{equation}
Taking $w:= u^m - \lambda \varphi$, with $\lambda >0$ and
$\varphi\in C^1(\overline{Q}_T)$, we get
\[
0\le \iint_{Q_T}(H -|\nabla (u^m-\lambda \varphi)|^{p-2}\nabla
(u^m-\lambda \varphi))\cdot \nabla\varphi dxdt.
\]
Letting $\lambda \rightarrow 0$ yields
\[
0\le \iint_{Q_T}(H -|\nabla u^m|^{p-2}\nabla u^m)\cdot\nabla\varphi dxdt.
\]
If in \eqref{h} we take $w:= u^m + \lambda \varphi$, with
$\lambda>0$, $\varphi \in C^1(\overline{Q}_T)$ and letting again
$\lambda\rightarrow 0$, then
\[
\iint_{Q_T}(H -|\nabla u^m|^{p-2}\nabla u^m)\cdot \nabla\varphi dxdt\le
0.
\]
Thus \eqref{uguaglianza} holds and \eqref{limite2} is proved.
\end{proof}
Obviously, the previous result holds also for a single equation. In
particular, we have the following corollary:
\begin{Corollary}\label{cor} Consider the problem
\begin{equation}\label{equazione}
 \begin{cases}
u_t-{\rm div}(|\nabla u^m|^{p-2}\nabla
u^m)=\left(a(x,t)-\!\int_{\Omega}\!K(\xi,t)u^2(\xi, t-\tau)d\xi\right)u^{p-1}, & \text{in } Q_T,\\
u(x, t) = 0,& \text{for }(x,t) \in  \partial \Omega \times (0,T), \\
u(\cdot,0)= u(\cdot,T),
\end{cases}
\end{equation}
and assume that
\begin{enumerate}
\item the exponents $p, m$  are such that $p \in (1,2)$ and $m
>p$,
\item the delay $\tau\in(0,+\infty)$,
\item the functions $a$ and $K$ belong to $L^\infty(Q_T)$, are
extended to $\Omega\times\R$ by $T$-periodicity and are non-negative
for a.a. $(x,t)\in Q_T$,
\item
 there exists a positive constant $C$
such that
 for all $\varepsilon>0$ and all the non-negative solutions $\eu$ of
\[
u=G_{\varepsilon} \big(1, f(u^+)\big),
\]
it results \[ \|\eu \|_{L^2(Q_T)}^2\le C.
\]
\end{enumerate}
Then problem \eqref{equazione} has a $T$-periodic non-negative and
non-trivial solution.
\end{Corollary}
Here $G_{\varepsilon} \big(1, f(u^+)\big)$ is defined as before.

\section{A priori bounds in $L^2(Q_T)$} \label{sec:aprioribounds}

In this section we apply Theorem~\ref{thm:generale} by looking for
explicit a priori bounds in $L^2(Q_T)$ for the solutions of the
approximating problem \eqref{2} in different situations.
More precisely, under different assumptions on the kernels $K_i,
i=1,2,3,4$ which model the interactions between the quantities
$u,v$, we determine the constants $C_1, C_2$ of
(\ref{eq:ipotesimaggiorazioni}) in an explicit form.
For this we consider two main different cases.
In the first one, which we call the ``coercive case'', we assume that $K_i(x,t)\ge \uk_i>0$ a.a. in
$Q_T$ for $i=1,4$.
In the second one, the ``non-coercive case'', we allow the non-negative functions $K_1,K_4$ to vanish on sets with
positive measure. We distinguish also between cooperative and
competitive situations by imposing sign conditions on $K_2,K_3$
having in mind the biological interpretation of model \eqref{1}.

\subsection{The coercive case}\label{subsec:coercive}

\begin{Theorem}\label{thm:coercivo}
Assume that
\begin{enumerate}
\item Hypotheses $\ref{ipotesi}$ are satisfied
\item
there are constants $\uk_i >0$, $i=1,4$,  such that
\[
K_i(x,t)\ge \uk_i \text{ for }i=1,4,
\]
for a.a. $(x,t)\in Q_T$, and $\uk_1\uk_4>\ok_2\ok_3$, where
$\ok_2,\ok_3$ are as in Hypothesis $\ref{ipotesi}$
\item Hypothesis $\ref{ipotesi1}.2$ is satisfied with
\begin{equation}\label{lecostanti}
\begin{aligned}
C_1= &\dfrac{T (\uk_4\|a\|_{L^\infty(Q_T)}
+\ok_2\|b\|_{L^\infty(Q_T)})}{\uk_1\uk_4-\ok_2\ok_3}\\
C_2= &\dfrac{T (\ok_3\|a\|_{L^\infty(Q_T)} +
\uk_1\|b\|_{L^\infty(Q_T)})} {\uk_1\uk_4-\ok_2\ok_3}.
\end{aligned}
\end{equation}
\end{enumerate}
Then problem \eqref{1} has a non-negative $T$-periodic solution
$(u,v)$ with non-trivial $u,v$.
\end{Theorem}
\begin{proof}
We just need to show that $\|\eu \|_{L^2(Q_T)}^2\le C_1$ and $\|\ev
\|_{L^2(Q_T)}^2\le C_2$ for any solution $(\eu ,\ev  )$ of
\eqref{eq:rhovarepsilon}. Then, assume $\eu \neq0$, thus $\eu  >0$ and
$\ev   \ge 0$ in $Q_T$ by Proposition~\ref{positivity}.
Multiplying the inequality
\[
\begin{aligned}
\dfrac{ \partial \eu }{ \partial t } - \varepsilon & \div( | \nabla \eu |^{p-2} \nabla \eu) - \div( | \nabla \eu^m |^{ p - 2 } \nabla \eu^m ) \\
\le & \left[ \|a\|_{ L^{\infty}( Q_{T} ) } - \int_{ \Omega } K_1( \xi, t ) \eu^{2}( \xi, t - \tau_{1} ) d\xi + \int_{ \Omega } K_2 ( \xi, t ) \ev^2( \xi, t - \tau_2 ) d\xi \right] \eu^{p-1}
\end{aligned}
\]
by $\eu $, integrating over $\Omega$ and using the Steklov averages
$(\eu )_h \in H^1(Q_{T- \delta})$, $\delta,h>0$, we obtain
\[
\begin{aligned}
& \frac{ \frac{1}{2} \frac{d}{dt} \int_{\Omega} ( \eu )_{h}^{2} dx + \varepsilon \int_{\Omega}| \nabla (\eu )_{h}|^{p} dx
+ m^{ p - 1 } \int_{\Omega} ( \eu )_{h}^{ (m-1) (p-1) } | \nabla (\eu  )_{h} |^{p} dx }{ \int_{\Omega} ( \eu )_{h}^{p} dx }\\
& \hphantom{HHHHH}\le \left(\|a\|_{L^\infty(Q_T)}-\int_{\Omega}K_1(\xi,t)(\eu)_h^2(\xi,t-\tau_1)d\xi
 +\int_{\Omega}K_2(\xi,t)(\ev  )_h^2(\xi,t-\tau_2)d\xi\right).
\end{aligned}
\]
Integrating the previous inequality over $[0, T]$, and passing to
the limit as $h\rightarrow 0$, by the $T$-periodicity of $\eu $, we
have that
\begin{equation}\label{equazione2}
0 < \left( T \|a\|_{ L^\infty( Q_{T} ) } - \uk_1 \| \eu \|_{ L^2( Q_{T} ) }^{2} + \ok_{2} \| \ev \|_{ L^2( Q_{T} ) }^{2}\right).
\end{equation}
The same procedure, when it is applied to the second
equation of \eqref{2}, leads to
\begin{equation}\label{equazione2bis}
0 < \left(T\|b\|_{L^\infty(Q_T)}-\uk_4\|\ev
\|_{L^2(Q_T)}^2+ \ok_3\|\eu \|_{L^2(Q_T)}^2\right).
\end{equation}
Hence, if $\eu \not\equiv 0$ and if $\ev \not\equiv
0$, by the positiveness of the right hand sides of \eqref{equazione2}
and \eqref{equazione2bis}, we have
\[
\begin{aligned}
\left(1-\dfrac{\ok_2\ok_3}{\uk_1\uk_4}\right)\|\eu\|_{L^2(Q_T)}^2
&<
\dfrac{T}{\uk_1}\left(\|a\|_{L^\infty(Q_T)}+\dfrac{\ok_2}{\uk_4}\|b\|_{L^\infty(Q_T)}\right)\\
\left(1-\dfrac{\ok_2\ok_3}{\uk_1\uk_4}\right)\|\ev\|_{L^2(Q_T)}^2
&<
\dfrac{T}{\uk_4}\left(\|b\|_{L^\infty(Q_T)}+\dfrac{\ok_3}{\uk_1}\|a\|_{L^\infty(Q_T)}\right)
\end{aligned}
\]
for any $\varepsilon \in (0,\varepsilon_0)$ and the desired bounds
follow since $\ok_2\ok_3<\uk_1\uk_4$. Obviously, if $\ev \equiv 0$,
then
\[
\|\eu\|_{L^2(Q_T)}^2\le \dfrac{T}{\uk_1}\|a\|_{L^\infty(Q_T)}\le
C_1.
\]
or if $\eu \equiv 0$, then
\[
\|\ev\|_{L^2(Q_T)}^2 \le \dfrac{T}{\uk_4}\|b\|_{L^\infty(Q_T)}\le C_2.   \qedhere
\]
\end{proof}

As an immediate consequence of the previous result we obtain the
following corollaries for the coercive-cooperative (see Remark~\ref{rem:onaprioriboundsandcooperation}) and the coercive-competitive cases respectively.

\begin{Corollary}\label{thm:coercivocooperativo}
Assume that
\begin{enumerate}
\item
Hypotheses $\ref{ipotesi}$ are satisfied with nontrivial coefficients $ a, b $,
\item $
0\le K_i(x,t) \text{ for }i=2,3,$ for a.a. $(x,t)\in Q_T$,
\item
there are constants $\uk_i >0$, $i=1,4$,  such that
\[
K_i(x,t)\ge \uk_i \text{ for }i=1,4,
\]
for a.a. $(x,t)\in Q_T$, and $\uk_1\uk_4>\ok_2\ok_3$, where
$\ok_2,\ok_3$ are as in Hypothesis $\ref{ipotesi}$.
\end{enumerate}
Then problem \eqref{1} has a non-negative $T$-periodic solution $(u,
v)$.
\end{Corollary}

\begin{Corollary}\label{thm:coercivocompetitivo}
Assume that
\begin{enumerate}
\item Hypotheses $\ref{ipotesi}$ are satisfied,
\item $
K_i(x,t)\le0 \text{ for }i=2,3,$ for a.a. $(x,t)\in Q_T$,
\item
there are constants $\uk_i >0$, $i=1,4$,  such that
\[
K_i(x,t)\ge \uk_i \text{ for }i=1,4,
\]
for a.a. $(x,t)\in Q_T$,
\item Hypothesis $\ref{ipotesi1}.2$ is satisfied with
\[
\begin{aligned}
C_1= &\dfrac{T}{\uk_1}\|a\|_{L^\infty(Q_T)}\\
C_2= &\dfrac{T}{\uk_4}\|b\|_{L^\infty(Q_T)}.
\end{aligned}
\]
\end{enumerate}
Then problem \eqref{1} has a non-negative $T$-periodic solution $(u,
v)$.
\end{Corollary}
We observe that the condition $\ok_2\ok_3<\uk_1\uk_4$ of
Theorem~\ref{thm:coercivo}  is crucial to establish the a priori
$L^2$-bounds on the solution pairs $(\eu ,\ev  )$ of \eqref{2}.
Roughly speaking this condition guarantees that the terms in the
equations that contribute to the growth of the respective species do
not prevail globally on those limiting the growth.

On the other hand, when the strict positivity of the functions $K_1$
and $K_4$ is relaxed, obtaining the needed a priori bounds becomes
more difficult (at least with our approach). In fact, we are able to
obtain simple a priori bounds in the non-coercive case when the
system is competitive, provided that $\min\{n(q-1), m(p-1)\}\;\ge
\;1$, i.e. when each equation of \eqref{1} is of slow or normal
diffusion type. Otherwise, we have to impose one more technical
restriction, i.e. $\min\displaystyle{\left\{m\frac{p-1}{p+1},
n\frac{q-1}{q+1}\right\}}\geq1$ to obtain a result like
Theorem~\ref{thm:coercivo} with no sign condition on the functions
$K_2$ and $K_3$.

Obviously, Theorem \ref{thm:coercivo} holds also for a single
equation. In particular, we have the following corollary:

\begin{Corollary}\label{cor1} Consider the problem
\eqref{equazione} and assume that
\begin{enumerate}
\item the exponents $p, m$  are such that $p \in (1,2)$ and $m
>p$,
\item the delay $\tau\in(0,+\infty)$,
\item the functions $a$ and $K$ belong to $L^\infty(Q_T)$, are
extended to $\Omega\times\R$ by $T$-periodicity and are non-negative
for a.a. $(x,t)\in Q_T$ and there exists a constant $\uk >0$ such
that
\[
K(x,t)\ge \uk
\]
for a.a. $(x,t)\in Q_T$,
\item
hypothesis $4$ of Corollary $\ref{cor}$ is satisfied with
\[C=\dfrac{T}{\uk}\|a\|_{L^\infty(Q_T)}.
\]
\end{enumerate}
Then problem \eqref{equazione} has a $T$-periodic non-negative and
non-trivial solution.
\end{Corollary}
\subsection{The non-coercive case: the competitive system}\label{subsec:weakcompetition}

\begin{Theorem}\label{thm:noncoerciveweakcompetitive}
Assume that
\begin{enumerate}
\item Hypotheses $\ref{ipotesi}$ are satisfied
\item $\ok_2=\ok_3=0$, that is
\[
K_i(x,t)\le 0 \text{ for }i=2,3,
\]
for a.a. $(x,t)\in Q_T$,
\item $m, n,p$ and $q$ are such that $m \; \ge  \;\displaystyle \frac{1}{p-1}$ and $n
\;\ge \; \displaystyle \frac{1}{q-1}$, i.e. both the equations of
system \eqref{1} have slow or normal diffusion,
\item Hypothesis $\ref{ipotesi1}.2$ is satisfied with
\[
C_1=\frac{|Q_T|^{\frac{m(p-1)-1}{(p-1)(m-1)}}}{\mu_p^{\frac{2}{(p-1)(m-1)}}}
\left(\frac{|\Omega|^{1-\frac{p}{2}}\|a\|_{L^\infty(Q_T)}(m(p-1)+1)^p
}{m^{p-1}p^p }\right)^{{{\frac{2}{(p-1)(m-1)}}}}\]and
\[
C_2=\frac{|Q_T|^{\frac{n(q-1)-1}{(q-1)(n-1)}}}{\mu_q^{\frac{2}{(q-1)(n-1)}}}
\left(\frac{|\Omega|^{1-\frac{q}{2}}\|b\|_{L^\infty(Q_T)}(n(q-1)+1)^q
}{n^{q-1}q^q }\right)^{\frac{2}{(q-1)(n-1)}}.
\]
\end{enumerate}
Then problem \eqref{1} has a $T$-periodic non-negative solution $(u,v)$ with non-trivial $u,v$.
\end{Theorem}
\begin{proof}
As a first step we find the bound for the non-negative solutions
$\eu$. Multiplying the first equation of \eqref{eq:rhovarepsilon} by
$\eu $, integrating in $Q_{T}$ and passing to the limit, as $h\to 0$, in the
Steklov averages $(\eu )_h$, we obtain
\begin{equation}\label{noncoercivecompetitive}
\begin{aligned}
m^{p-1}\left(\frac{p}{m(p-1)+1} \right)^p \Int \left|\nabla\eu
^{\frac{m(p-1)+1}{p}}\right|^p dxdt &\le \varepsilon \Int |\nabla
\eu|^p dxdt+ \Int |\nabla \eu^m|^{p-2}\nabla \eu^m \nabla \eu
dxdt\\& \le \|a\|_{L^\infty(Q_T)}\|\eu \|_{L^p(Q_T)}^p \le
|\Omega|^{1-\frac{p}{2}} \|a\|_{L^\infty(Q_T)}\|\eu \|_{L^2(Q_T)}^p,
\end{aligned}
\end{equation}
by the $T$-periodicity of $\eu $ and the non-positivity of the
function $K_2$. Using the H\"{o}lder inequality with $r:=
\frac{m(p-1)+1}{2}$, and the Poincar\'{e} inequality, one has:
\[
\begin{aligned}
\Int\eu ^2 dxdt &\le |Q_T|^{\frac{m(p-1)-1}{m(p-1)+1}}\left(\Int \eu
^{m(p-1)+1}dxdt\right)^{\frac{2}{m(p-1)+1}} =
|Q_T|^{\frac{m(p-1)-1}{m(p-1)+1}} \left\|\eu
^{\frac{m(p-1)+1}{p}}\right\|_{L^p(Q_T)}^{\frac{2p}{m(p-1)+1}}\\&
\le |Q_T|^{\frac{m(p-1)-1}{m(p-1)+1}}\left(\frac{1}{\sqrt[p]{\mu_p}}
\left\|\nabla \eu
^{\frac{m(p-1)+1}{p}}\right\|_{L^p(Q_T)}\right)^{\frac{2p}{m(p-1)+1}}.
\end{aligned}
\]
Thus by (\ref{noncoercivecompetitive}) we get
\[
\begin{aligned}
\|\eu \|_{L^2(Q_T)}^2&\le
|Q_T|^{\frac{m(p-1)-1}{m(p-1)+1}}\left(\frac{1}{\sqrt[p]{\mu_p}}
\left\|\nabla \eu
^{\frac{m(p-1)+1}{p}}\right\|_{L^p(Q_T)}\right)^{\frac{2p}{m(p-1)+1}}\\&
\le
\frac{|Q_T|^{\frac{m(p-1)-1}{m(p-1)+1}}}{\mu_p^{\frac{2}{m(p-1)+1}}}
\left(\frac{|\Omega|^{1-\frac{p}{2}}\|a\|_{L^\infty(Q_T)}(m(p-1)+1)^p
}{m^{p-1}p^p }\right)^{\frac{2}{m(p-1)+1}}\|\eu
\|_{L^2(Q_T)}^{\frac{2p}{m(p-1)+1}}.
\end{aligned}\]
This implies
\[
\|\eu \|_{L^2(Q_T)}^2 \le
\frac{|Q_T|^{\frac{m(p-1)-1}{(p-1)(m-1)}}}{\mu_p^{\frac{2}{(p-1)(m-1)}}}
\left(\frac{|\Omega|^{1-\frac{p}{2}}\|a\|_{L^\infty(Q_T)}(m(p-1)+1)^p
}{m^{p-1}p^p }\right)^{\frac{2}{(p-1)(m-1)}}.
\]
In an analogous way we obtain that
\[
\|\ev  \|_{L^2(Q_T)}^2\le
\frac{|Q_T|^{\frac{n(q-1)-1}{(q-1)(n-1)}}}{\mu_q^{\frac{2}{(q-1)(n-1)}}}
\left(\frac{|\Omega|^{1-\frac{q}{2}}\|b\|_{L^\infty(Q_T)}(n(q-1)+1)^q
}{n^{q-1}q^q }\right)^{\frac{2}{(q-1)(n-1)}},
\]
if $\ev  $ is a solution of the second equation of
\eqref{eq:rhovarepsilon}.
\end{proof}
The previous result still holds for a single equation:
\begin{Corollary}\label{cor2} Consider problem
\eqref{equazione} and assume that
\begin{enumerate}
\item the exponents $p, m$  are such that $p \in (1,2)$ and $m\;
 \ge  \;\frac{1}{p-1}$,
\item the delay $\tau\in(0,+\infty)$,
\item the functions $a$ and $K$ belong to $L^\infty(Q_T)$, are
extended to $\Omega\times\R$ by $T$-periodicity and are non-negative
for a.a. $(x,t)\in Q_T$,
\item
hypothesis 4 of Corollary $\ref{cor}$ is satisfied with
\[C=\frac{|Q_T|^{\frac{m(p-1)-1}{(p-1)(m-1)}}}{\mu_p^{\frac{2}{(p-1)(m-1)}}}
\left(\frac{|\Omega|^{1-\frac{p}{2}}\|a\|_{L^\infty(Q_T)}(m(p-1)+1)^p
}{m^{p-1}p^p }\right)^{{\frac{2}{(p-1)(m-1)}}}.
\]
\end{enumerate}
Then problem \eqref{equazione} has a $T$-periodic non-negative and
non-trivial solution.
\end{Corollary}
\subsection{The non-coercive case: $\boldsymbol{\min\displaystyle{\left\{m\frac{p-1}{p+1},
n\frac{q-1}{q+1}\right\}}\geq1}$}\label{subsec:bellecostanti}

In the case that $\min\displaystyle{\left\{m\frac{p-1}{p+1},
n\frac{q-1}{q+1}\right\}}\geq1$, we are able to find explicit bounds
(although complicated) without any assumption on the sign of the
functions $K_2, K_3$, as shown in the next result.
\begin{Theorem}\label{thm:noncoercivebruttecostanti}
Assume
\begin{enumerate}
\item $\min\displaystyle{\left\{m\frac{p-1}{p+1},
n\frac{q-1}{q+1}\right\}}>1$,
\item
$K_i(x,t)\ge 0$, $i=1,4$ and $K_i(x,t)\le\ok_i$, $i=2,3$ for a.a.
$(x,t)\in Q_T$ and for some positive constants $\ok_i$, $i=2,3$,
\item
Hypothesis $\ref{ipotesi1}.2$ is satisfied with
\begin{equation}\label{eq:bellecostanti}
\begin{aligned}
C_1=
&\sqrt{T}\left\{\frac{(q-1)(n-1)(p-1)(m-1)}{(q-1)(n-1)(p-1)(m-1)-4}\alpha_p
+ \beta_p^\frac{(q-1)(n-1)(p-1)(m-1)}{(q-1)(n-1)(p-1)(m-1)-4}\right\}^{1/2},\\
C_2=
&\sqrt{T}\left\{\frac{(q-1)(n-1)(p-1)(m-1)}{(q-1)(n-1)(p-1)(m-1)-4}\alpha_q
+
\beta_q^\frac{(q-1)(n-1)(p-1)(m-1)}{(q-1)(n-1)(p-1)(m-1)-4}\right\}^{1/2},
\end{aligned}
\end{equation}
\end{enumerate}
where
\[
\begin{aligned}
\alpha_p:&=
C_p^{\frac{(p-1)(m-1)+2}{(p-1)(m-1)}}\left(2T\|a\|_{L^\infty(Q_T)}^2
\right)^{\frac{2}{(p-1)(m-1)}}
\\&+C_p^{\frac{(p-1)(m-1)+2}{(p-1)(m-1)}}\left(2\ok_2^2C_q^{\frac{(q-1)(n-1)+2}{(q-1)(n-1)}}\right)^{\frac{2}{(p-1)(m-1)}}\left(2T\|b\|_{L^\infty(Q_T)}^2
\right)^{\frac{4}{(q-1)(n-1)(p-1)(m-1)}},
\end{aligned}
\]
\[
\begin{aligned}
\alpha_q:&=
C_q^{\frac{(q-1)(n-1)+2}{(q-1)(n-1)}}\left(2T\|b\|_{L^\infty(Q_T)}^2
\right)^{\frac{2}{(q-1)(n-1)}}
\\&+C_q^{\frac{(q-1)(n-1)+2}{(q-1)(n-1)}}\left(2\ok_3^2C_p^{\frac{(p-1)(m-1)+2}{(p-1)(m-1)}}\right)^{\frac{2}{(q-1)(n-1)}}\left(2T\|a\|_{L^\infty(Q_T)}^2
\right)^{\frac{4}{(q-1)(n-1)(p-1)(m-1)}},
\end{aligned}
\]
\[
\beta_p:=
C_p^{\frac{(p-1)(m-1)+2}{(p-1)(m-1)}}\left(2\ok_2^2C_q^{\frac{(q-1)(n-1)+2}{(q-1)(n-1)}}\right)^{\frac{2}{(p-1)(m-1)}}\left(2\ok_3^2\right)^{\frac{4}{(q-1)(n-1)(p-1)(m-1)}},
\]
and
\[
\beta_q:=
C_q^{\frac{(q-1)(n-1)+2}{(q-1)(n-1)}}\left(2\ok_3^2C_p^{\frac{(p-1)(m-1)+2}{(p-1)(m-1)}}\right)^{\frac{2}{(q-1)(n-1)}}\left(2\ok_2^2\right)^{\frac{4}{(q-1)(n-1)(p-1)(m-1)}}.
\]
Here
\begin{equation}\label{eq:bellecostanti1}
C_p:=\left(\frac{\left[(p-1)(m-1)+2\right]^p
|\Omega|^{\frac{(p-1)(m-1)}{2}}}{
p^pm^{p-1}(3-p)\mu_p}\right)^{\frac{4}{(p-1)(m-1)+2}}T^{\frac{(p-1)(m-1)-2}{(p-1)(m-1)+2}}
\end{equation}
and
\begin{equation}\label{eq:bellecostanti1bis}
 C_q:=
\left(\frac{\left[(q-1)(n-1)+2\right]^q
|\Omega|^{\frac{(q-1)(n-1)}{2}}}{
q^qn^{q-1}(3-q)\mu_q}\right)^{\frac{4}{(q-1)(n-1)+2}}T^{\frac{(q-1)(n-1)-2}{(q-1)(n-1)+2}}.
\end{equation}
Then problem \eqref{1} has a non-negative $T$-periodic solution
$(u,v)$ with non-trivial $u,v$.
\end{Theorem}
\begin{proof}
Let $(\eu ,\ev)$ be a solution of \eqref{eq:rhovarepsilon}. We have, by
the Poincar\'{e} inequality and  the H\"{o}lder inequality with $r:=
\frac{(p-1)(m-1) +2}{2}$,
\[
 \left(\int_\Omega \eu ^2dx\right)^{\frac{(p-1)(m-1) +2}{2}} \le |\Omega|^{\frac{(p-1)(m-1)}{2}}
 \int_\Omega\eu ^{(p-1)(m-1)+2}dx\le \frac{1}{ \mu_p}
|\Omega|^{\frac{(p-1)(m-1)}{2}}
 \int_\Omega\left|\nabla\eu ^{\frac{(p-1)(m-1)+2}{p}}\right|^pdx.
\]
Integrating over $[0,T]$, we have:
\begin{equation}\label{3.28bis}
\begin{aligned}
\int_0^T \left(\int_\Omega \eu ^2dx\right)^{\frac{(p-1)(m-1) +2}{2}}
dt \le \frac{1}{ \mu_p} |\Omega|^{\frac{(p-1)(m-1)}{2}}
 \iint_{Q_T}\left|\nabla\eu ^{\frac{(p-1)(m-1)+2}{p}}\right|^pdxdt.
\end{aligned}
\end{equation} Now, we estimate the term $ \displaystyle
\iint_{Q_T}\left|\nabla\eu ^{\frac{(p-1)(m-1)+2}{p}}\right|^pdxdt$.
Multiplying the first equation of \eqref{eq:rhovarepsilon} by
$\eu^{3-p} $, integrating in $Q_T$ and passing to the limit in the
Steklov averages $(\eu )_h$, we obtain by the $T$-periodicity of
$\eu $
\[
\begin{aligned}
&m^{p-1}(3-p)\left[\frac{p}{(p-1)(m-1)+2}\right]^p\iint_{Q_T}\left|\nabla\eu ^{\frac{(p-1)(m-1)+2}{p}}\right|^pdx dt \\
\le &\int_0^T\left[\|a\|_{L^\infty(Q_T)}
 +\ok_2\int_\Omega\ev  ^2(\xi,t-\tau_2)d\xi\right]\left(\int_\Omega\eu ^2dx\right)dt \\
\le &\left[\int_0^T\left(\|a\|_{L^\infty(Q_T)}
 +\ok_2\int_\Omega\ev  ^2(\xi,t-\tau_2)d\xi\right)^2dt\right]^{\frac{1}{2}}
 \left[\int_0^T\left(\int_\Omega\eu ^2dx\right)^2dt\right]^{\frac{1}{2}}.
\end{aligned}
\]
Thus
\begin{equation}\label{3.29bis}
\begin{aligned}
\iint_{Q_T}\left|\nabla\eu ^{\frac{(p-1)(m-1)+2}{p}}\right|^pdx dt
\le M_p\left[\int_0^T\left(\|a\|_{L^\infty(Q_T)}
 +\ok_2\int_\Omega\ev  ^2(\xi,t-\tau_2)d\xi\right)^2dt\right]^{\frac{1}{2}}
 \left[\int_0^T\left(\int_\Omega\eu^2 dx
\right)^2dt\right]^{\frac{1}{2}},
\end{aligned}
\end{equation}
where $M_p:= \displaystyle
\frac{1}{m^{p-1}(3-p)}\left[\frac{(p-1)(m-1)+2}{p}\right]^p$.
 By the H\"{o}lder inequality with $s:= \frac{(p-1)(m-1)+2}{4}$ (observe that $s \ge1$ by the assumption on $m$ and $p$)
and by \eqref{3.28bis}, \eqref{3.29bis}, it follows
\[
\begin{aligned}
&\int_0^T\left(\int_\Omega\eu ^2dx\right)^2dt \le
T^{\frac{(p-1)(m-1)-2}{(p-1)(m-1)+2}}\left(\int_0^T\left(\int_\Omega\eu^2dx\right)^{\frac{(p-1)(m-1)+2}{2}}dt \right)^{\frac{4}{(p-1)(m-1)+2}}\\
& \le C_p\left[\int_0^T\left(\|a\|_{L^\infty(Q_T)}
 +\ok_2\int_\Omega\ev  ^2(\xi,t-\tau_2)d\xi\right)^2dt\right]^{\frac{2}{(p-1)(m-1)+2}}
 \left[\int_0^T\left(\int_\Omega\eu ^2 dx
 \right)^2dt\right]^{\frac{2}{(p-1)(m-1)+2}},
\end{aligned}
\]
where $C_p$ is the constant defined in \eqref{eq:bellecostanti1}.
Therefore, setting $U=\int_0^T(\int_\Omega\eu ^2dx)^2dt$,
$V=\int_0^T(\int_\Omega\ev ^2dx)^2dt$, and using the assumption $m
>\displaystyle \frac{p+1}{p-1}$, the last inequality implies
\[
\begin{aligned}
U&\le C_p^{\frac{(p-1)(m-1)+2}{(p-1)(m-1)}}
\left[\int_0^T\left(\|a\|_{L^\infty(Q_T)}
 +\ok_2\int_\Omega\ev  ^2dx\right)^2dt\right]^{\frac{2}{(p-1)(m-1)}} \\
&\le C_p^{\frac{(p-1)(m-1)+2}{(p-1)(m-1)}}
\left[2T\|a\|_{L^\infty(Q_T)}^2
 +2\ok_2^2V\right]^{\frac{2}{(p-1)(m-1)}}\\
&\le C_p^{\frac{(p-1)(m-1)+2}{(p-1)(m-1)}}
\left[\left(2T\|a\|_{L^\infty(Q_T)}^2 \right)^{\frac{2}{(p-1)(m-1)}}
+\left(2\ok_2^2V\right)^{\frac{2}{(p-1)(m-1)}}\right].
\end{aligned}
\]
In an analogous way, we can show that
\[
\begin{aligned}
V \le
C_q^{\frac{(q-1)(n-1)+2}{(q-1)(n-1)}}\left[\left(2T\|b\|_{L^\infty(Q_T)}^2
\right)^{\frac{2}{(q-1)(n-1)}}
+\left(2\ok_3^2U\right)^{\frac{2}{(q-1)(n-1)}}\right],
\end{aligned}
\]
where $C_q$ is the constant introduced in
\eqref{eq:bellecostanti1bis}. Hence, it results
\[
\begin{aligned}
U& \le
C_p^{\frac{(p-1)(m-1)+2}{(p-1)(m-1)}}\left[\left(2T\|a\|_{L^\infty(Q_T)}^2
\right)^{\frac{2}{(p-1)(m-1)}}
+\left(2\ok_2^2V\right)^{\frac{2}{(p-1)(m-1)}}\right]\\& \le
C_p^{\frac{(p-1)(m-1)+2}{(p-1)(m-1)}}\left(2T\|a\|_{L^\infty(Q_T)}^2
\right)^{\frac{2}{(p-1)(m-1)}}
\\&+C_p^{\frac{(p-1)(m-1)+2}{(p-1)(m-1)}}\left(2\ok_2^2C_q^{\frac{(q-1)(n-1)+2}{(q-1)(n-1)}}\right)^{\frac{2}{(p-1)(m-1)}}\left(2T\|b\|_{L^\infty(Q_T)}^2
\right)^{\frac{4}{(q-1)(n-1)(p-1)(m-1)}}\\&
+C_p^{\frac{(p-1)(m-1)+2}{(p-1)(m-1)}}\left(2\ok_2^2C_q^{\frac{(q-1)(n-1)+2}{(q-1)(n-1)}}\right)^{\frac{2}{(p-1)(m-1)}}\left(2\ok_3^2U\right)^{\frac{4}{(q-1)(n-1)(p-1)(m-1)}}.
\end{aligned}
\]
The last inequality has the form:
\begin{equation}\label{riusata}
U \le \alpha + \beta U^{\frac{4}{(q-1)(n-1)(p-1)(m-1)}},
\end{equation}
with $\alpha, \beta>0$. Since $\min\left\{m\frac{p-1}{p+1},
n\frac{q-1}{q+1}\right\}>1$ the function $f(U):=\alpha+\beta
U^{\frac{4}{(q-1)(n-1)(p-1)(m-1)}}$ is strictly concave, and then
\begin{equation}\label{U}
U\le f(U)  \le f(U_0)+ f'(U_0)(U-U_0),
\end{equation}
where $U_0:=
\beta^\frac{(q-1)(n-1)(p-1)(m-1)}{(q-1)(n-1)(p-1)(m-1)-4}$. Using
the fact that $f(U_0)= \alpha + U_0$ and \eqref{U}, one has
\[
\begin{aligned} U&\le
\frac{(q-1)(n-1)(p-1)(m-1)}{(q-1)(n-1)(p-1)(m-1)-4}\alpha +
\beta^\frac{(q-1)(n-1)(p-1)(m-1)}{(q-1)(n-1)(p-1)(m-1)-4}.
\end{aligned}\] A final application of H\"older's inequality shows that $\|\eu
\|_{L^2}^2\le T^{1/2}U^{1/2}=C_1$. The argument for $\ev  $ proceeds
in a similar way.
\end{proof}
As a consequence of Theorem \ref{subsec:bellecostanti}  one has the
next corollaries for the cooperative and competitive cases,
respectively.
\begin{Corollary}\label{thm:coercivocooperativo1bis}
Assume that
\begin{enumerate}
\item
 $\min\displaystyle{\left\{m\frac{p-1}{p+1},
n\frac{q-1}{q+1}\right\}}>1$,
\item
$ K_i(x,t)\ge 0 \text{ for }i=1,4, $ for a.a. $(x,t)\in Q_T$, and
there are positive constants $\ok_2, \ok_3$ such that
\[
0\le K_i(x,t)\le\ok_i \text{ for }i=2,3,
\]
for a.a. $(x,t)\in Q_T$,
\item
Hypothesis $\ref{ipotesi1}.2$ is satisfied with $C_1$ and $C_2$ as
in \eqref{eq:bellecostanti}.
\end{enumerate}
Then problem \eqref{1} has a non-negative $T$-periodic solution
$(u,v)$ with non-trivial $u,v$.
\end{Corollary}

\begin{Corollary}\label{thm:coercivocompetitivo1bis}
Assume that
\begin{enumerate}
\item
$\min\displaystyle{\left\{m\frac{p-1}{p+1},
n\frac{q-1}{q+1}\right\}}>1$,
\item
$ K_i(x,t)\ge 0 \text{ for }i=1,4, $ for a.a. $(x,t)\in Q_T$, and
there are non-negative constants $\uk_2, \uk_3$ such that
\[
 -\uk_i\le
K_i(x,t)\le 0 \text{ for }i=2,3,
\]
for a.a. $(x,t)\in Q_T$,
\item
Hypothesis $\ref{ipotesi1}.2$ is satisfied with
\[
C_1=\left(TC_p^{\frac{(p-1)(m-1)+2}{(p-1)(m-1)}}\left(2T\|a\|_{L^\infty(Q_T)}^2
\right)^{\frac{2}{(p-1)(m-1)}}\right)^{\frac{1}{2}} \] and
\[ C_2=\left(TC_q^{\frac{(q-1)(n-1)+2}{(q-1)(n-1)}}\left(2T\|b\|_{L^\infty(Q_T)}^2
\right)^{\frac{2}{(q-1)(n-1)}}\right)^{\frac{1}{2}},
\]
where $C_p$ and $C_q$ are as in \eqref{eq:bellecostanti1} and in
\eqref{eq:bellecostanti1bis}.
\end{enumerate}
Then problem \eqref{1} has a non-negative $T$-periodic solution $(u,
v)$.
\end{Corollary}

The proof of Theorem \ref{thm:noncoercivebruttecostanti} suggests
the following result when
$\min\displaystyle{\left\{m\frac{p-1}{p+1},
n\frac{q-1}{q+1}\right\}}=1.$

\begin{Theorem}\label{normaldiff}
Suppose that Assumptions $2$ and $3$ of Theorem
$\ref{thm:noncoercivebruttecostanti}$ hold true. If, in addition,
$$\min\displaystyle{\left\{m\frac{p-1}{p+1},
n\frac{q-1}{q+1}\right\}}=1$$
 and
\begin{equation}\label{sopra}
C_p^{\frac{(p-1)(m-1)+2}{(p-1)(m-1)}}\left(2\ok_2^2C_q^{\frac{(q-1)(n-1)+2}{(q-1)(n-1)}}\right)^{\frac{2}{(p-1)(m-1)}}
\left(2\ok_3^2\right)^{\frac{4}{(q-1)(n-1)(p-1)(m-1)}} <1,
\end{equation}
then problem \eqref{1} has a non-negative $T$-periodic solution $(u,
v)$.
\end{Theorem}
\begin{proof}
First note that, if , for instance $n\frac{q-1}{q+1}=1$, then
$(q-1)(n-1)=2$, so that the expression in \eqref{sopra} can be
simplified. Now the proof proceeds as the one of Theorem
$\ref{thm:noncoercivebruttecostanti}$ up to inequality
\eqref{riusata}, which now reads
\[
U \le \alpha + \beta U,
\]
where $\beta$ is the left hand side of \eqref{sopra}. Since
$\beta<1$, we obtain the desired upper bound on $U$.
\end{proof}

\begin{rem}
Observe that the technique used to prove Theorem \ref{thm:generale}
(or Corollary \ref{cor}),  and the a priori estimates in $L^2(Q_T)$
can be adapted to prove analogous results if we consider system
\eqref{1} with $p, q \ge 2$, that is if we consider a double
degeneracy (or a single degenerate equation) as in \cite{fnp}, but
with a $p$-linear term in the right hand side.
\end{rem}

\section{A generalization in the competitive case}\label{generalizzazione}
The techniques used in the previous sections allow us to prove the
existence of a $T$-periodic non-negative solution $(u,v)$ with
non-trivial $u,v$ for the following system:
\begin{equation}\label{system1}
 \begin{cases}
 \begin{aligned}
u_t-\div(|\nabla u^m|^{p-2}\nabla
u^m)\!\!=\!\left(a(x,t)-\!\int_{\Omega}\!K_1(\xi,t)u^{\alpha}(\xi, t-\tau_1)d\xi\!+\!\int_{\Omega}\!K_2(\xi,t)v^{\alpha}(\xi, t-\tau_2)d\xi\right)u^{p-1}, \quad \text{in } Q_T,\\
v_t-\div(|\nabla v^n|^{q-2}\nabla
v^n)\!\!=\!\left(b(x,t)+\!\int_{\Omega}\!K_3(\xi,t)u^{\alpha}(\xi, t-\tau_3)d\xi-\!\int_{\Omega}\!K_4(\xi,t)v^{\alpha}(\xi, t-\tau_4)d\xi\right)v^{q-1},  \quad\; \text{in } Q_T,\\
 \end{aligned}\\
u(x, t) = v(x,t)=0,\qquad\qquad \qquad \qquad\qquad \qquad\qquad  \;\; \text{for }(x,t) \in  \partial \Omega \times (0,T), \\
u(\cdot,0)= u(\cdot,T)\text{ and }v(\cdot,0)= v(\cdot,T), \\
\end{cases}
\end{equation}
where $\alpha \ge 1$,  $K_i(t,x)\le 0$ ($i=2,3$), and $m, n, p, q$,
$\tau_i$ ($i=1,2,3,4$), $a, b$ and $K_i$ ($i=1,4$), are as in
Hypotheses \ref{ipotesi}.
\subsection{The coexistence theorem}
As before, one can prove that Lemma \ref{compact} and Proposition
\ref{positivity} still hold for the associated nondegenerate singular $p$-Laplacian problem
\begin{equation}\label{system2}
\begin{cases}
\begin{aligned}
&\dfrac{\partial u}{\partial t}- \varepsilon \div(|\nabla u
|^{p-2}\nabla u) -\div(|\nabla u^m |^{p-2}\nabla u^m) = \left(
a(x,t) - \int_{\Omega}K_1(\xi,t)u^\alpha(\xi, t-\tau_1)d\xi +\right.
\\&
\left. \qquad \qquad \qquad \qquad \qquad \qquad \qquad \qquad\qquad
\qquad \qquad \qquad +\int_{\Omega}K_2(\xi,t)v^\alpha(\xi,
t-\tau_2)d\xi
\right) u^{p-1}, \quad \text{in } Q_T,\\
&\dfrac{\partial v}{\partial t}- \varepsilon \div(|\nabla
v|^{q-2}\nabla v) -\div(|\nabla v^n |^{q-2}\nabla v^n) = \left(
b(x,t) + \int_{\Omega}K_3(\xi,t)u^\alpha(\xi, t-\tau_3)d\xi -\right.
\\&
\left. \qquad \qquad \qquad \qquad \qquad \qquad \qquad \qquad
\qquad \qquad \qquad \qquad - \int_{\Omega}K_4(\xi,t)v^\alpha(\xi,
t-\tau_4)d\xi\right) v^{q-1},  \quad \text{in } Q_T,
\end{aligned}\\
u(\cdot,t)|_{\partial \Omega}= v(\cdot,t)|_{\partial \Omega}=0, \quad\text{for a.a.}\; t \in (0,T),\\
u (\cdot,0)=u(\cdot,T)\text{ and }v(\cdot,0)=v(\cdot,T),
\end{cases}
\end{equation}
where $\varepsilon>0$ is small enough. Moreover the next result
holds:
\begin{Proposition}\label{systemgeneralized}
There is a constant $R>0$ such that
\[
\|\eu  \|_{L^\infty(Q_T)},\|\ev   \|_{L^\infty(Q_T)}<R
\]
for all solution pairs $(\eu  ,\ev   )$ of
\[
(u,v)=G_{\varepsilon} \big(1, f(u^+,v^+), g(u^+,v^+)\big),
\]
and all $\varepsilon >0$. In
particular, one has that
\[
\deg\big((u,v)-G_{\varepsilon}\big(1, f(u^+,v^+),
g(u^+,v^+)\big),B_R,0\big)=1.
\]
Here $G_\varepsilon$ is defined as in Section
$\ref{regularizedproblem}$,
\[ f(u^+,v^+):=
\left(a-\int_{\Omega}K_1(\xi,\cdot)\,(u^+)^\alpha(\xi,
\cdot-\tau_1)d\xi+\int_{\Omega}K_2(\xi,\cdot)\,(v^+)^\alpha(\xi,
\cdot-\tau_2)d\xi\right)(u^+)^{p-1}
\]
and
\[
g(u^+,v^+):=\left(b+\int_{\Omega}K_3(\xi,\cdot)\,(u^+)^\alpha(\xi,
\cdot-\tau_3)d\xi-\int_{\Omega}K_4(\xi,\cdot)\,(v^+)^\alpha(\xi,
\cdot-\tau_4)d\xi\right)(v^+)^{q-1},
\]
\end{Proposition}
\begin{proof}
By the first equation of \eqref{system2}, we have
\[
\dfrac{\partial \eu  }{\partial t}-\varepsilon\div(|\nabla
\eu|^{p-2}\nabla \eu)- \div(\nabla \eu^m|\nabla \eu^m |^{p-2})\le
K\eu^{p-1}  ,
\]
where $K:=\|a\|_{L^\infty(Q_T)}$. By Lemma \ref{limitatezza}, using the
Steklov averages $(\eu  )_h\in H^1(Q_{T- \delta})$, $\delta,h>0$,
and the fact that $(\eu)_h$ converges to $\eu$ in $L^\infty(Q_T)$,
we conclude that $\|\eu \|_{L^\infty(Q_T)}\le R_1$ for some $R_1>0$
independent of $\varepsilon$.  Analogously, $\|\ev
\|_{L^\infty(Q_T)}\le R_2$ for some constant $R_2>0$. Therefore it
is enough to choose $R>\max\{R_1, R_2\}$.

The second part of the Proposition follows as in the proof of Proposition \ref{R-moser}
\end{proof}

From now on we make the following assumption:
\begin{Assumption}\label{ipotesi2} The functions $a$ and $b$ are
such that \[\displaystyle \min
\left\{\dfrac{1}{T}\iint_{Q_T}a(x,t)e_p^p(x)dxdt-\dfrac{\uk_2C_2}{T},
\dfrac{1}{T}\iint_{Q_T}b(x,t)e_q^q(x)dxdt-\dfrac{\uk_3C_1}{T}\right\}>0.
\]
Here $\uk_2, \uk_3$ are as in Hypotheses \ref{ipotesi}.3, $\mu_p$,
$\mu_q$, $e_p$  and $e_q$ are as in Section
\ref{regularizedproblem}.
\end{Assumption}
As before, take $\varepsilon$ in $(0,\varepsilon_0)$, where $\varepsilon_0$
is such that
\[
\begin{aligned}
\theta(C_1, C_2):= \min&\left\{\frac{1}{T}\Int  a(x,t)e_p^p(x)
dxdt -\varepsilon_0 \mu_p-\dfrac{\uk_2C_2}{T}, \right.\\
&\left. \frac{1}{T}\Int b(x,t)e_q^q(x) dxdt -\varepsilon_0
\mu_q-\dfrac{\uk_3C_1}{T}\right\} >0.
\end{aligned}
\]
Then, Lemma \ref{stima-gradiente} and Proposition \ref{r} still hold
with
\[
M_1 := \left(\frac{\|a\|_{L^\infty(Q_T)}
|Q_T|^{\frac{1}{\beta'}}\left(
\frac{1}{\mu_p}\right)^{\frac{1}{\beta}}[(p-1)(m-1)-s)]^p}{
[m(p-1)]^{p-1}[(p-1)(m-p)-ps)] }\right)^{\frac{\beta}{p(\beta-1)}},
\]
\[
M_2 := \left(\frac{\|b\|_{L^\infty(Q_T)}
|Q_T|^{\frac{1}{\delta'}}\left(
\frac{1}{\mu_q}\right)^{\frac{1}{\delta}}[(q-1)(n-1)-s)]^q}{
[n(q-1)]^{q-1}[(q-1)(n-q)-qs)]
}\right)^{\frac{\delta}{q(\delta-1)}},
\]
and
\[
\begin{aligned}
r_0:=\min &\left\{\left(\dfrac{\Int  a(x,t) e_p^p(x)dxdt -
\varepsilon_0T\mu_p}{D_1}\right)^{\frac{1}{\alpha}}\!,
\left(\dfrac{\Int
 a(x,t)e_p^p(x) dxdt -
\varepsilon_0T\mu_p}{D_1}\right)^{\frac{1}{s}}\!,\right.\\
&\left. \quad\left(\dfrac{\Int b(x,t)e_q^q(x) dxdt -
\varepsilon_0T\mu_q}{D_2}\right)^{\frac{1}{\alpha}}\!,
\left(\dfrac{\Int
 b(x,t)e_q^q(x) dxdt -
\varepsilon_0T\mu_q}{D_2}\right)^{\frac{1}{s}}\right\},
\end{aligned}
\]
where $\beta$, $\beta'$, $\delta$,$\delta'$, $D_1$ and $D_2$ are as
in Section \ref{regularizedproblem}. Proceeding as in Theorem
\ref{thm:generale}, one can prove that the next coexistence result
holds:
\begin{Theorem}\label{thm:generale1}
 Assume that there exist two positive constants $C_1,C_2$ such that
 for all $\varepsilon>0$ and all  solution pairs $(\eu  ,\ev   )$ of
\eqref{system2} it results
\[
 \|\eu \|_{L^\alpha(Q_T)}^\alpha\le
C_1\text{ and }\|\ev \|_{L^\alpha(Q_T)}^\alpha\le C_2.
\]
Then, problem \eqref{system1} has a $T$-periodic non-negative
solution $(u,v)$ with non-trivial $u,v$.
\end{Theorem}
Obviously, the previous result holds also for a single equation. In
particular, we have the following corollary:
\begin{Corollary}\label{corollary} Consider the problem
\begin{equation}\label{equazione1bis}
 \begin{cases}
\displaystyle u_t-{\rm div}(|\nabla u^m|^{p-2}\nabla
u^m)=\left(a(x,t)-\!\int_{\Omega}\!K(\xi,t)u^\alpha(\xi, t-\tau)d\xi\right)u^{p-1}, & \text{in } Q_T,\\
u(x, t) = 0,& \text{for }(x,t) \in  \partial \Omega \times (0,T), \\
u(\cdot,0)= u(\cdot,T),
\end{cases}
\end{equation}
and assume that
\begin{enumerate}
\item the exponents $p, m$  are such that $p \in (1,2)$ and $m
>p$,
\item the delay $\tau\in(0,+\infty)$,
\item the functions $a$ and $K$ belong to $L^\infty(Q_T)$, are
extended to $\Omega\times\R$ by $T$-periodicity and are non-negative
for a.a. $(x,t)\in Q_T$,
\item
 there exists a positive constant $C$
such that
 for all $\varepsilon>0$ and all solutions $\eu$ of
\[
u=G_{\varepsilon} \big(1, f(u^+)\big),
\]
it results \[ \|\eu \|_{L^\alpha(Q_T)}^\alpha\le C.
\]
\end{enumerate}
Then problem \eqref{equazione1bis} has a $T$-periodic non-negative
and non-trivial solution.
\end{Corollary}
\subsection{A priori bounds in $\boldsymbol{L^\alpha(Q_T)}$: the coercive and the non coercive cases}\label{sec:aprioribounds1}
In this subsection we apply Theorem \ref{thm:generale1} by looking
for explicit a priori bounds in $L^\alpha(Q_T)$ for the solutions of
\eqref{system2} in the ''coercive case" and in the ''non-coercive
case".
\begin{Theorem}\label{thm:coercivo1}
Assume that
\begin{enumerate}
\item Hypotheses $\ref{ipotesi}$ are satisfied,
\item
there exist constants $\uk_i >0$, $i=1,4$,  such that
\[
K_i(x,t)\ge \uk_i \text{ for }i=1,4,
\]
for a.a. $(x,t)\in Q_T$,
\item
$ K_i(x,t)\le 0 \text{ for }i=2,3, $ for a.a. $(x,t)\in Q_T$,
\item Hypothesis $\ref{ipotesi1}.2$ is satisfied with
\[
\begin{aligned}
C_1= &\dfrac{T}{\uk_1}\|a\|_{L^\infty(Q_T)}
\\
C_2= &\dfrac{T}{\uk_4}\|b\|_{L^\infty(Q_T)}.
\end{aligned}
\]
\end{enumerate}
 Then problem \eqref{1} has a non-negative $T$-periodic solution
$(u,v)$ with non-trivial $u,v$.
\end{Theorem}
\begin{proof}
We just need to show that $\|\eu \|_{L^\alpha(Q_T)}^\alpha\le C_1$
and $\|\ev \|_{L^\alpha(Q_T)}^\alpha\le C_2$ for any solution $(\eu
,\ev )$ of \eqref{system2}. Proceeding as in Theorem
\ref{thm:coercivo}, one has
\[
\varepsilon\int_0^T\frac{\int_{\Omega}|\nabla\eu |^p
dx}{\left(\int_{\Omega}\eu ^2 dx\right)^{\frac{p}{2}}}dt \le
|\Omega|^{1-\frac{p}{2}}\left(T\|a\|_{L^\infty(Q_T)}-\uk_1\|\eu
\|_{L^\alpha(Q_T)}^\alpha\right)
\]
and again
\[
\varepsilon\lambda_{p,2} T\le
|\Omega|^{1-\frac{p}{2}}\left(T\|a\|_{L^\infty(Q_T)}-\uk_1\|\eu
\|_{L^\alpha(Q_T)}^\alpha\right).
\]
 The same procedure, when applied to the second
equation of \eqref{2}, leads to
\[
\varepsilon \lambda_{q,2} T\le
|\Omega|^{1-\frac{q}{2}}\left(T\|b\|_{L^\infty(Q_T)}-\uk_4\|\ev
\|_{L^\alpha(Q_T)}^\alpha\right),
\]
 Hence, we have
\[
\begin{aligned}
\|\eu\|_{L^\alpha(Q_T)}^\alpha &\le
\dfrac{T\|a\|_{L^\infty(Q_T)}}{\uk_1}\,,\\
\|\ev\|_{L^\alpha(Q_T)}^\alpha &\le
\dfrac{T\|b\|_{L^\infty(Q_T)}}{\uk_4}.
\end{aligned}
\]
\end{proof}

\begin{Theorem}\label{thm:noncoerciveweakcompetitive1}
Assume that
\begin{enumerate}
\item Hypotheses $\ref{ipotesi}$ are satisfied,
\item $ K_i(x,t)\le 0 \text{ for }i=2,3, $ for a.a. $(x,t)\in Q_T$,
\item Hypothesis $\ref{ipotesi1}.2$ is satisfied with
\[
C_1=\frac{|Q_T|}{\mu_p^{\frac{\alpha}{m(p-1)}}}
\left(\frac{\|a\|_{L^\infty(Q_T)}(m(p-1)+\alpha)^p }{\alpha
m^{p-1}p^p }\right)^{\frac{\alpha}{m(p-1)}}\]and
\[
C_2=\frac{|Q_T|}{\mu_q^{\frac{\alpha}{n(q-1)}}}
\left(\frac{\|b\|_{L^\infty(Q_T)}(n(q-1)+\alpha)^q }{\alpha
n^{q-1}q^q }\right)^{\frac{\alpha}{n(q-1)}}.
\]
\end{enumerate}
Here $\ok_2, \ok_3,$ are as in Hypotheses $\ref{ipotesi}$. Then
problem \eqref{1} has a $T$-periodic non-negative solution $(u,v)$
with non-trivial $u,v$.
\end{Theorem}
\begin{proof}
We just need to show that $\|\eu \|_{L^\alpha(Q_T)}^\alpha\le C_1$
and $\|\ev \|_{L^\alpha(Q_T)}^\alpha\le C_2$ for any solution $(\eu
,\ev )$ of \eqref{system2}. Multiplying the first equation of
\eqref{system2} by $\eu^{\alpha-p+1} $, integrating in $Q_{T}$ and
passing to the limit, as $h\to0$, in the Steklov averages $(\eu )_h$, we obtain,
as in Proposition \ref{R-moser},
\[
\begin{aligned}
m^{p-1}\alpha\left(\frac{p}{m(p-1)+\alpha} \right)^p \Int
\left|\nabla\eu ^{\frac{m(p-1)+\alpha}{p}}\right|^pdxdt &\le
\varepsilon \alpha\Int u^{\alpha-1} |\nabla \eu|^pdxdt \\&+ \Int
|\nabla \eu^m|^{p-2}\nabla \eu^m \nabla \eu^\alpha dxdt\\& \le
\|a\|_{L^\infty(Q_T)}\|\eu \|_{L^\alpha(Q_T)}^\alpha,
\end{aligned}\]
by the $T$-periodicity of $\eu $ and the non-positivity of the
function $K_2$. Using the H\"{o}lder inequality, with $r:=
\frac{m(p-1)+\alpha}{\alpha}$, and the Poincar\'{e} inequality, one
has:
\[
\begin{aligned}
\Int\eu ^\alpha dxdt &\le
|Q_T|^{\frac{m(p-1)}{m(p-1)+\alpha}}\left(\Int \eu
^{m(p-1)+\alpha}dxdt\right)^{\frac{\alpha}{m(p-1)+\alpha}} =
|Q_T|^{\frac{m(p-1)}{m(p-1)+\alpha}}\left\|\eu
^{\frac{m(p-1)+\alpha}{p}}\right\|_{L^p(Q_T)}^{\frac{\alpha
p}{m(p-1)+\alpha}}\\& \le
|Q_T|^{\frac{m(p-1)}{m(p-1)+\alpha}}\left(\frac{1}{\sqrt[p]{\mu_p}}
\left\|\nabla \eu
^{\frac{m(p-1)+\alpha}{p}}\right\|_{L^p(Q_T)}\right )^{\frac{\alpha
p}{m(p-1)+\alpha}}.
\end{aligned}
\]
Thus
\[
\begin{aligned}
\|\eu \|_{L^\alpha(Q_T)}^\alpha&\le
|Q_T|^{\frac{m(p-1)}{m(p-1)+\alpha}}\left(\frac{1}{\sqrt[p]{\mu_p}}
\left\|\nabla \eu
^{\frac{m(p-1)+\alpha}{p}}\right\|_{L^p(Q_T)}\right)^{\frac{\alpha
p}{m(p-1)+\alpha}}\\& \le
\frac{|Q_T|^{\frac{m(p-1)}{m(p-1)+\alpha}}}{\mu_p^{\frac{\alpha}{m(p-1)+\alpha}}}
\left(\frac{\|a\|_{L^\infty(Q_T)}(m(p-1)+\alpha)^p }{\alpha
m^{p-1}p^p }\right)^{\frac{\alpha}{m(p-1)+\alpha}}\|\eu
\|_{L^\alpha(Q_T)}^{\frac{\alpha^2}{m(p-1)+\alpha}}. \end{aligned}\]
This implies
\[
\|\eu \|_{L^\alpha(Q_T)}^\alpha \le
\frac{|Q_T|}{\mu_p^{\frac{\alpha}{m(p-1)}}}
\left(\frac{\|a\|_{L^\infty(Q_T)}(m(p-1)+\alpha)^p }{\alpha
m^{p-1}p^p }\right)^{\frac{\alpha}{m(p-1)}}.
\]
In an analogous way we obtain that
\[
\|\ev  \|_{L^\alpha(Q_T)}^\alpha\le
\frac{|Q_T|}{\mu_q^{\frac{\alpha}{n(q-1)}}}
\left(\frac{\|b\|_{L^\infty(Q_T)}(n(q-1)+\alpha)^q }{\alpha
n^{q-1}q^q }\right)^{\frac{\alpha}{n(q-1)}},
\]
if $\ev  $ is a solution of the second equation of \eqref{system2}.
\end{proof}
An immediate consequence of Theorems \ref{thm:coercivo1} and
\ref{thm:noncoerciveweakcompetitive1} is the following existence
result for a single equation:
\begin{Corollary} Consider the problem
\eqref{equazione1bis} and assume that
\begin{enumerate}
\item the exponents $p, m$  are such that $p \in (1,2)$ and $m
>p$,
\item the delay $\tau\in(0,+\infty)$,
\item the functions $a$ and $K$ belong to $L^\infty(Q_T)$, are
extended to $\Omega\times\R$ by $T$-periodicity and are non-negative
for a.a. $(x,t)\in Q_T$, \item hypothesis $4$ of Corollary
$\ref{corollary}$ is satisfied with
\[
C=\dfrac{T}{\uk}\|a\|_{L^\infty(Q_T)}
\]
if $K(t,x)\ge\underline{k} >0$, or
\[
C= \frac{|Q_T|}{\mu_p^{\frac{\alpha}{m(p-1)}}}
\left(\frac{\|a\|_{L^\infty(Q_T)}(m(p-1)+\alpha)^p }{\alpha
m^{p-1}p^p }\right)^{\frac{\alpha}{m(p-1)}}
\]
if $K(t,x) \ge0$.
\end{enumerate}
Then problem \eqref{equazione1bis} has a $T$-periodic non-negative
and non-trivial solution.
\end{Corollary}
\begin{proof}
We just need to show that $\|\eu \|_{L^\alpha(Q_T)}^\alpha\le C$ for
a positive constant $C$. Proceeding  as in Theorem
\ref{thm:coercivo1} if $K(t,x)\ge\underline{k} >0$ or in Theorem
\ref{thm:noncoerciveweakcompetitive1} if $K(t,x) \ge0$, one has that
$\|\eu \|_{L^\alpha(Q_T)}^\alpha\le
\dfrac{T}{\uk_1}\|a\|_{L^\infty(Q_T)}$ or $\|\eu
\|_{L^\alpha(Q_T)}^\alpha \le
\frac{|Q_T|}{\mu_p^{\frac{\alpha}{m(p-1)}}}
\left(\frac{\|a\|_{L^\infty(Q_T)}(m(p-1)+\alpha)^p }{\alpha
m^{p-1}p^p }\right)^{\frac{\alpha}{m(p-1)}}$, respectively. The
thesis follows by Corollary \ref{corollary}.
\end{proof}

\begin{rem}
We underline the fact that the generalization presented in this
section can be extended to a single degenerate equation or to a
double degenerate system, namely when $p, q \ge 2$, as considered in
\cite{fnp}.
\end{rem}


\end{document}